\documentclass[12pt]{amsart}

\usepackage[arrow,matrix,curve]{xy}

\usepackage[dvips]{graphicx} 

\usepackage{amssymb, latexsym, amsmath, amscd, array, hyperref, epigraph
}

\newtheorem{theorem}{Theorem}[section]

\theoremstyle{definition}
\newtheorem{definition}[theorem]{Definition}

\numberwithin{equation}{section}

\newcommand\N {{\mathbb N}}

\newcommand\R {{\mathbb R}} 
\newcommand\C {{\mathbb C}} 

\newcommand\PP {{\mathcal P}}

\newcommand\astr {{}^{\ast}\hspace{-0.6pt}{\mathbb R}}

\newcommand\astc{{{}^\ast\C}} 

\newcommand\asr {{}^{\ast}\hspace{-1pt}r}

\newcommand\astq {{}^{\ast}\hspace{-0.3pt}{\mathbb Q}}

\newcommand\asts {{}^{\ast}\hspace{-1.85pt}S}

\newcommand\astn {{}^{\ast}\hspace{-0.8pt}{\mathbb N}}

\newcommand\astv {{}^{\ast}\hspace{-1pt}V}

\newcommand\astA{{}^{\ast}\hspace{-3.6pt}A}

\newcommand\astb{{}^{\ast}\hspace{-2.4pt}B}

\newcommand\astx{{}^{\ast}\hspace{-2.4pt}X}

\newcommand\astf {{}^{\ast}\hspace{-2.6pt}f}

\newcommand\astin {{}^{\ast}\hspace{-5pt}\in}

\newcommand\Q {{\mathbb Q}}

\newcommand\st{\textbf{st}} 

\newcommand\sh{\mathbf{sh}}

\newcommand\inte{{\textbf{int}}}

\author[P.F.]{Peter Fletcher}\address{P. Fletcher, School of Computing
and Mathematics, Keele University, Keele, Staffordshire, ST5 2BG, UK}
\email{p.fletcher@keele.ac.uk}

\author[K. H.]{Karel Hrbacek} \address{K. Hrbacek, 305 West 98 Street,
Apt 4AS, New York, NY 10025, U.S.A.} \email{khrbacek@icloud.com}

\author[V. K.]{Vladimir Kanovei} \address{V. Kanovei, Laboratory 6,
IPPI RAN, Bolshoy Karetny per.\;19, build.\;1, Moscow 127051, Russia;
and MIIT, 9b9 Obrazcova St, 127994 Moscow, Russia}
\email{kanovei@googlemail.com}

\author[M. K.]{Mikhail G. Katz}\address{M. Katz, Department of
Mathematics, Bar Ilan University, Ramat Gan 52900 Israel}
\email{katzmik@macs.biu.ac.il}

\author[C. L.]{Claude Lobry}\address{C. Lobry, Centre de Recherche en
Histoire des Id\'ees (CRHI), Universit\'e de Nice Sophia Antipolis,
98, Bd \'Edouard Herriot, BP 3209 - 06204 Nice Cedex}
\email{claude.lobry@icloud.com}

\author[S. S.]{Sam Sanders}\address{S. Sanders,
Ludwig-Maximilians-Universit\"at M\"unchen, Fakult\"at f\"ur
Philosophie, Wissenschaftstheorie und Religionswissenschaft, Munich
Center for Mathematical Philosophy, Geschwister-Scholl-Platz 1,
D-80539 M\"unchen, Germany} \email{sasander@me.com}

\begin{document}

\thispagestyle{empty}


\title[Approaches to analysis with infinitesimals]{Approaches to
analysis with infinitesimals following Robinson, Nelson, and others}

\begin{abstract}
This is a survey of several approaches to the framework for working
with infinitesimals and infinite numbers, originally developed by
Abraham Robinson in the 1960s, and their constructive engagement with
the Cantor--Dedekind postulate and the \emph{Intended Interpretation}
hypothesis.  We highlight some applications including (1)\;Loeb's
approach to the Lebesgue measure, (2)\;a radically elementary approach
to the vibrating string, (3)\;true infinitesimal differential
geometry.  We explore the relation of Robinson's and related
frameworks to the multiverse view as developed by Hamkins.

Keywords: axiomatisations, infinitesimal, nonstandard analysis,
ultraproducts, superstructure, set-theoretic foundations, multiverse,
naive integers, intuitionism, soritical properties, ideal elements,
protozoa.
\end{abstract}

\maketitle \tableofcontents

\section{Of sets and automobiles: a parable}
\label{wings}

The last third of the 19th century saw (at least) two dynamic
innovations: \emph{set theory} and \emph{the automobile}.  Consider
the following parable.

A silvery family sedan is speeding down the highway.  It enters heavy
traffic.  Every \emph{epsilon} of the road requires a new \emph{delta}
of patience on the part of the passengers.  The driver's inquisitive
daughter Sarah,%
\footnote{For the literary career of this character see
Kanovei--Katz--Schaps \cite{15a}.}
sitting in the front passenger seat, discovers a mysterious switch in
the glove compartment.  With a click, the sedan spreads wings and
lifts off above the highway congestion in an \emph{infinitesimal}
instant.  Soon it is a mere silvery speck in an \emph{infinite}
expanse of the sky.  A short while later it lands safely on the front
lawn of the family's home.

Sarah's cousin Georg%
\footnote{This character similarly made a (cameo) appearance in the
final section of Kanovei--Katz--Schaps \cite[pp.\;17--18]{15a}.}
refuses to believe the story: true Sarah's father is an NSA man, but
everybody knows that Karl Benz's 1886 \emph{Patent-Motorwagen} had no
wing option!  

At a less parabolic level, some mathematicians feel that, on the one
hand, ``It is quite easy to make mistakes with infinitesimals, etc.''
(Quinn \cite[p.\;31]{Qu12}) while, as if by contrast, ``Modern
definitions are completely selfcontained, and the only properties that
can be ascribed to an object are those that can be rigorously deduced
from the definition\ldots{} Definitions that are modern in this sense
were developed in the late 1800s.''  (ibid., p.\;32).

We will have opportunity to return to Sarah, cousin Georg,
\emph{switches}, and the heroic ``late 1800s'' in the sequel; see in
particular Section~\ref{s63}.

\section{Introduction}
\label{s1}

The framework created by Abraham Robinson in the 1960s and called by him 
``nonstandard analysis'' is an active
research area featuring many applications in other fields,%
\footnote{\label{f3}See Section\;\ref{s7b} and Stroyan--Luxemburg
\cite{SL}, Henle--Kleinberg \cite{HK}, Arkeryd\;\cite{Ar81},
Lutz--Goze \cite{LG}, Keisler \cite{Ke84}, Albeverio et
al.\;\cite{Al86}, Wallet \cite{Wa86}, van den Berg\;\cite{Va87},
Callot \cite{Ca93}, Rubio \cite{Ru94}, Diener--Diener \cite{DD},
Henson \cite{He97}, Ross \cite{Ro97}, Jin \cite{Ji02},
Kanovei--Shelah\;\cite{KS}, Kanovei--Lyubetskii \cite{KL}, Goldbring
\cite{Go10}, van den Dries--Goldbring\;\cite{VG}, Nowik--Katz
\cite{15d}, Pra\v z\'ak--Slav\'\i k \cite{PS}; this list of
applications is by no means comprehensive.}
its own journal (\emph{Journal of Logic and Analysis}),%
\footnote{See \url{http://www.logicandanalysis.org}}
and high-profile advocates like Terry Tao;%
\footnote{\label{f5}The field has also had its share of high-profile
detractors like Errett Bishop \cite{Bi77} and Alain Connes \cite{CLS}.
Their critiques were analyzed in Katz--Katz \cite{11a},
Katz--Leichtnam \cite{13d}, Kanovei--Katz--Mormann \cite{13c}, and
Sanders \cite{Sa17}; see also material in Section~\ref{s5} around
note~\ref{f29}.  For further details on Connes and the taming of
differential geometry see note\;\ref{f24}.  Additional criticisms were
voiced by J.\;Earman \cite{Ea75}, K.\;Easwaran \cite{Ea14},
H.\;M.\;Edwards \cite{Ed07}, G.\;Ferraro \cite{Fe04},
J.\;Grabiner\;\cite{Gr81}, J.\;Gray\;\cite{Gr15}, P.\;Halmos
\cite{Ha85}, H.\;Ishiguro \cite{Is90}, K.\;Schubring \cite{Sc05},
Y.\;Sergeyev \cite{Se15}, and D.\;Spalt \cite{Sp02}.  These were dealt
with respectively in Katz--Sherry \cite{13f}, Bascelli et
al.\;\cite{14a}, Kanovei--Katz--Sherry \cite{15b}, Bair et
al.\;\cite{17a}, Borovik--Katz \cite{12b}, B\l{}aszczyk et
al.\;\cite{17d}, B\l{}aszczyk et al.\;\cite{16b}, Bascelli et
al.\;\cite{16a}, B\l aszczyk et al.\;\cite{17e}, Gutman et
al.\;\cite{17g}, Katz--Katz \cite{11b}.}
see e.g., Tao \cite{Ta14}, Tao--Vu \cite{TV}.  The time may be ripe
for a survey of some of the approaches to the field.

\subsection{Audience}

The text presupposes some curiosity about infinitesimals in general
and Robinson's framework in particular.  While the text is addressed
to a somewhat informed audience not limited to specialists in the
field, an elementary introduction is provided in Section~\ref{s2}.

Professional mathematicians and logicians curious about Robinson's
framework, as well as mathematically informed philosophers are one
possible (proper) subset of the intended audience, as are physicists,
engineers, and economists who seem to have few inhibitions about using
terms like \emph{infinitesimal} and \emph{infinite number}.%
\footnote{See e.g., a quotation in Delfini--Lobry \cite{DL} from
\emph{Berkeley Physics Course}, Crawford \cite{Cr66} and
Section~\ref{s71s}.}

\subsection{The CD+II mindset}
\label{s11}

A survey of this sort necessarily enters into a relationship of
constructive (if not subversive) engagement with a pair of assumptions
commonly held among mathematicians, namely
\begin{enumerate}
\item
the Cantor--Dedekind postulate (CD) identifying the line in physical
space with the real number line (see Ehrlich\;\cite{Eh06} and
Katz--Katz \cite{12c}), and
\item
the \emph{Intended Interpretation} hypothesis (II), entailing an
identification of a standard~$\mathbb{N}$ in its set-theoretic
context, on the one hand, with ordinary intuitive counting numbers, on
the other.
\end{enumerate}
We will deal with the II in more detail in Section~\ref{s23}.  How the
engagement with the CD+II mindset plays itself out should become
clearer in the sequel.

\subsection{Summary of perspectives}

This article was inspired in part by the posting of Joel Hamkins at
Math Overflow~\cite{Ha15b}.  Hamkins wrote: ``There are at least three
distinct perspectives one can naturally take on when undertaking work
in nonstandard analysis.  In addition, each of these perspectives can
be varied in two other dimensions, independently. Those dimensions
are, first, the order of nonstandardness (whether one wants
nonstandardness only for objects, or also for functions and
predicates, or also for sets of functions and sets of those and so
on); and second, how many levels of standardness and nonstandardness
one desires.''

We shall describe the three perspectives and discuss their
similarities and differences in Sections~\ref{s2} through~\ref{s4}.
The breakdown into three main perspectives parallels that in
Bair--Henry \cite{BH08}.  Sections~\ref{s7b} through~\ref{s7} explore
additional aspects of Robinson's framework.  The survey
Benci--Forti--Di Nasso \cite{Be06} was an earlier effort in this
general direction.

\section{Construction perspective}
\label{s2}

Hamkins wrote:%
\footnote{The quotations in this section have been slightly edited.}
``In this perspective, one thinks of the nonstandard universe as the
result of an explicit construction. In the most basic instance, one
starts with the familiar ordered field of real numbers
$\langle\R,+,\times,<, 0,1\rangle$ and extends it to a larger ordered
field.''  

The larger field then necessarily contains infinitesimals and
infinitely large numbers.  Constructing such an extension by itself is
easy,%
\footnote{The simplest example of such an extension is the field of
rational functions (quotients of polynomials with real coefficients),
linearly ordered by defining~$f < g$ if and only if~$f(x) < g(x)$ for
all sufficiently large~$x$.  Then real numbers are embedded in it as
constant functions, the function~$f(x) = \frac{1}{x}$ represents an
infinitesimal, and the function~$g(x) = x$ represents an infinitely
large number.  More advanced examples are provided by Levi-Civita
fields; see Lightstone--Robinson \cite{LR}.}
but not just any extension will do. In order for it to serve as a tool
for analysis, one needs at least to be able to extend every
function~$f$ of interest defined on~$\mathbb{R}$ to a function~$\astf$
defined on the larger field so that
$\astf\!\!\!\downharpoonright_{\,\R}\;=f$.  This extension is expected
to preserve the important properties of~$f$ (see Section~\ref{s25} for
an introductory discussion of the Transfer Principle).  We will sketch
an ultrapower construction of such a hyperreal
extension~$\R\hookrightarrow\astr$.

\subsection
{Ideal points in projective geometry} 
\label{s31}

As a motivational comment, it may be helpful to compare extending the
field~$\R$ to adding ideal points at infinity in projective geometry.%
\footnote{Such an analogy is meaningful in the context of
\emph{ordered} fields, whereas the addition of points at infinity in
projective geometry makes sense for an arbitrary field.}
In projective geometry, an affine plane defined by the
equation~$ax+by+cz=d$ is outfitted with an assortment of ideal points
at infinity, one for each pencil of parallel lines in the plane.  The
new points are viewed as satisfying the ``same'' equation once one
introduces homogeneous coordinates~$ax_1+bx_2+cx_3=dx_4$.  In addition
to ideal points being added to the ambient space, each substructure,
namely each line, is similarly enriched by the addition of one of
these points.  This is analogous to a hyperreal extension presented
below where not only the extension~$\astr$ itself has additional
points, but every object, such as subset, or function, is similarly
enriched.

At the beginning of the 17th century Johannes Kepler invoked a
\emph{principle of continuity} to justify (at least) two distinct
procedures: a continuous sweep of all conics aided by ideal points at
infinity (a prototype of the modern theorem that all conics are
projectively equivalent), and a view of a circle (or more general
curves) as an infinite-sided polygon.  The latter view found an avid
adherent in the person of Gottfried Leibniz, whose \emph{law of
continuity} postulated that the rules valid in the finite realm remain
valid in the infinite realm; see Section~\ref{s85} and Katz--Sherry
\cite{13f} for additional details.

\subsection{The ultrapower construction}
\label{s24}

Let~$\R^{\N}$ denote the ring of infinite sequences of real numbers,
with arithmetic operations defined termwise.  Then we have a totally
ordered field
\begin{equation}
\label{52}
\astr=\R^{\N}\!/\text{MAX}
\end{equation}
where ``MAX'' is a suitable maximal ideal.  To produce such a maximal
ideal, one can exploit a \emph{finitely additive} measure~$\xi$,
\[
\xi\colon\mathcal{P}(\N)\to\{0,1\}
\]
(thus~$\xi$ takes only two values,~$0$ and~$1$) taking the value~$1$
on each cofinite set, where~$\mathcal{P}(\N)$ is the set of subsets
of~$\N$.  For each pair of complementary subsets of~$\N$, such a
measure~$\xi$ ``decides'' in a coherent way which one is
``negligible'' (i.e., of measure~$0$) and which is ``dominant''
(measure~$1$).

The ideal MAX consists of all ``negligible'' sequences~$\langle
u_n\rangle$, i.e., sequences which vanish for a set of indices of
measure~$1$, namely,
\[
\xi\big(\{n\in\N\colon u_n=0\}\big)=1.
\]
The subset~$U_\xi\subseteq\mathcal{P}(\N)$ whose members are sets of
measure~$1$ is called a \emph{free} (or \emph{nonprincipal})
ultrafilter.

Note the formal analogy between \eqref{52} and the construction of the
real numbers as equivalence classes of Cauchy sequences of rational
numbers.  In both cases, the subfield is embedded in the superfield by
means of the constant sequences, and the ring of sequences is factored
by a \emph{maximal ideal}.

Elements of~$\astr$ are called \emph{hyperreal numbers}.  The
field~$\R$ is embedded into~$\astr$ via a mapping that assigns to
each~$r \in \R$ the hyperreal~$\asr$, namely, the equivalence class of
the constant sequence with value~$r$; we shall identify~$r$ and
$\asr$.  The equivalence classes of sequences with terms
from~$\mathbb{N}$ form the set~$\astn$ of \emph{hypernatural} numbers.
The equivalence class of the sequence~$\langle n\colon
n\in\mathbb{N}\rangle$ (respectively,~$\langle\frac{1}{n}\colon
n\in\mathbb{N}\rangle$) is an infinitely large integer (respectively,
an infinitesimal).

The \emph{order}~$\,{}^{\ast}\!\!\!<$ on~$\astr$ is defined by setting
\[
[\langle u_n\rangle] \;{}^{\ast}\!\!\!< [\langle v_n\rangle]
\text{\quad if and only if\quad} \xi(\{n\in\N\colon u_n<v_n\})=1.
\]

This order extends the ordering~$<$ of~$\R$ as follows: for~$r , s \in
\R$, we have~$r<s$ if and only if~$r \;{}^{\ast}\!\!\!< s$ (that
is,~$\asr \;{}^{\ast}\!\!\!< {}^{\ast}\!s$). Hence we can drop the
asterisk on~$<$ with no risk of confusion.

\emph{Addition}~${}^{\ast}\!+$ on~$\astr$ is defined by
\[
[\langle u_n\rangle]\;{}^{\ast}\hspace{-5pt}+[\langle v_n\rangle] =
[\langle u_n+v_n\rangle]
\]
and similarly for multiplication.  These operations extend the
corresponding operations on~$\R$, and hence the asterisks can and
shall be dropped.  Moreover,~$\astr$, with these operations and
ordering, is an ordered field extending~$\R$.

The equivalence classes of sequences with rational terms form the
subfield~$\astq$ consisting of \emph{hyperrational} numbers.

A hyperreal number~$x$ is \emph{finite} (or \emph{limited}) if~$-r < x
< r~$ for some~$r\in\R$, and \emph{infinitesimal} if~$-r<x < r~$ for
all~$r\in\R$,~$r>0$.  Every finite hyperreal number~$x$ rounds off to
the nearest real number, called its \emph{shadow} (or \emph{standard
part}) and denoted~$\sh(x)$; here the difference~$x-\mathbf{sh} (x)$
is infinitesimal.

Similarly, the ultrapower construction allows one to extend arbitrary
functions and relations from~$\R$ to~$\astr$.  Every
function~$f\colon\R\to\R$ has a \emph{natural extension}
$\astf\colon\astr\to\astr$ acting componentwise:~$\astf([\langle
u_n\rangle])=[\langle f( u_n) \rangle ]$.  If~$S \subseteq \R$, then
the \emph{natural extension} of~$S$, denoted~$\asts$, is defined by
\[
[\langle u_n\rangle] \in \asts \text{\; if and only if \;}
\xi\left(\{n\in\N\colon u_n^{\phantom{I}}\in S\}\right)=1;
\]
similarly for relations on~$\R$.  The asterisks are habitually dropped
when there is no risk of confusion; typically they are dropped for
functions and operations, but not for sets.

One of the first treatments of an ultrapower-type construction
appeared in Hewitt \cite{Hew48}.

\subsection{Transfer Principle}
\label{s25}

The \emph{Transfer Principle} is a type of theorem that, depending on
the context, asserts that rules, laws or procedures valid for a
certain number system, still apply (i.e., are ``transfered'') to an
extended number system.  Thus, the familiar
extension~$\Q\hookrightarrow\R$ preserves the property of being an
ordered field.  To give a negative example, the
extension~$\R\hookrightarrow\R\cup\{\pm\infty\}$ of the real numbers
to the so-called \emph{extended reals} does not preserve the field
properties.  the extension~$\R\hookrightarrow\C$ preserves the field
axioms, but does not preserve the property of not having a square root
of~$-1$.

The hyperreal extension~$\R\hookrightarrow\astr$
preserves \emph{all} first-order properties.  The result in essence
goes back to \L o\'s\;\cite{Lo55}.

For example, the identity~$\sin^2 x+\cos^2 x=1$ remains valid for all
hyperreal~$x$, including infinitesimal and infinite
inputs~$x\in\astr$.  Another example of a transferable property is the
property that
\[
\text{for all positive }x,y, \text{ if } x<y \text{ then }
\frac{1}{y}<\frac{1}{x}.
\]
The Transfer Principle applies to formulas like that characterizing
the continuity of a function~$f\colon\R\to\R$ at a point~$c\in\R$:
\[
(\forall\varepsilon > 0)(\exists\delta > 0)(\forall
x)\big[\,|x-c|<\delta \;\Rightarrow\; |f(x)-f(c)|<\varepsilon\big];
\]
namely, formulas that quantify over \emph{elements} of~$\R$.  See
Lindstr\o m\;\cite{Li88}, Goldblatt \cite{Go98}, and
Gordon--Kusraev--Kutateladze \cite{GKK} for additional details.

\subsection{Elementary applications}

Let~$H$ be an infinite hypernatural number (more formally,
$H\in\astn\setminus\N$) and~$z\in\C$.  We retrieve formulas of the
sort that already appeared in Euler; see Bair et al.\;\cite{17a}.  In
the following we will exploit hypercomplex numbers~$\astc=\astr +
\sqrt{-1} \; \astr$.  For example we obtain
\[
e^z\approx\left(1+\tfrac{z}{H}\right)^H
\]
where~$\approx$ is the relation of infinite proximity.%
\footnote{Here~$z\approx w$ if and only if~$|z-w|$ is infinitesimal.}
In particular, the identity~$e^{i\pi}=-1$ takes the form
\begin{equation}
\label{e22}
\left(1+\tfrac{i\pi}{H}\right)^H\approx -1.
\end{equation}
Since the principal branch of~$z\mapsto H\sqrt[H]{z}$ is Lipschitz
near~$z=-1$, it preserves the relation of infinite proximity there.%
\footnote{Here we use the principal branch of the root.  The point is
that the Lipschitz constant for the internal function in question can
be chosen finite, since the Lipschitz constant does not depend on the
index assuming that we work in a small neighborhood of~$-1$.
Therefore the Lipschitz property with the same constant holds by the
Transfer Principle (see Section~\ref{s25}) for all infinite~$H$, as
well.

Alternatively, we have the primitive root
$\sqrt[H]{-1}=\cos\frac{\pi}{H}+i\sin\frac{\pi}{H}$, hence
$H\!\sqrt[H]{-1}=H\cos\frac{\pi}{H}+iH\sin\frac{\pi}{H} \simeq H +
i\pi$ since~$\cos \frac{\pi}{H} \simeq 1$ and
$\frac{H}{\pi}\sin\frac{\pi}{H}\simeq1$.}
Therefore the relation~\eqref{e22} implies
\[
H\sqrt[H]{-1}\approx H+i\pi
\]
yielding the following formula for~$\pi$:
\[
\pi\approx \frac{H\sqrt[H]{-1}-H}{\sqrt{-1}}.
\]
This expression for~$\pi$ can be interpreted as the derivative of
$(-1)^x$ at~$x=-\frac{1}{2}$.%
\footnote{The derivative is computed from the definition, using the
infinitesimal~$\frac{1}{H}$.  The derivative of~$(-1)^x$
is~$(\ln(-1))(-1)^x$.  We substitute~$x=-\frac{1}{2}$ to obtain~$\pi$
since~$\ln(-1)=\pi i$ and~$(-1)^{-1/2}=-i$.}

To mention another elementary application, the heuristic principle
that the period of small oscillations of the pendulum is independent
of their amplitude finds precise mathematical implementation for
oscillations with infinitesimal amplitude Kanovei--Katz--Nowik
\cite{16c}.  A transparent treatment of the Jordan curve theorem
exploiting Robinson's framework appears in Kanovei--Reeken
\cite{KR98}.  For additional applications see Section~\ref{s7b}.%
\footnote{See also the references provided in note~\ref{f3}.}

\subsection{Compactness, saturation}

We have described the ultrapower construction with respect to a fixed
free ultrafilter~$U$ on the natural numbers.  Another way to implement
these ideas is to use the compactness theorem of first-order logic, as
Robinson originally did in \cite{Ro66}.

In some applications of nonstandard analysis one may wish to employ
stronger saturation properties (see below) than those satisfied by
\eqref{52}; this can usually be accomplished by using special
ultrafilters on larger index sets.

The property of \emph{saturation} of a hyperreal field is analogous to
compactness in classical analysis, and can be expressed as follows.

\begin{definition}
\label{f11}
A hyperreal field~$\astr$ is \emph{countably saturated} if every
nested infinite sequence of nonempty internal sets (see
Sections~\ref{s35} and~\ref{s51}) has a common element.
\end{definition}

More generally, a structure is called~$\kappa$-saturated if for every
collection of properties (expressible in the language of the
structure) of cardinality strictly less than~$\kappa$, each finite
subcollection of which is satisfied by some element of the structure,
there is an element of the structure that satisfies all of the
properties simultaneously.  Note that~\eqref{52} is always countably
saturated.%
\footnote{Traditionally, countable saturation is the same as
$\aleph_1$-saturation, because a collection of cardinality strictly
less than~$\aleph_1$ is the same as a countable (including finite,
which is a trivial case anyway) collection.}

The existence of a common point for a decreasing nested sequence of
compact sets~$\langle K_n\colon n\in\N\rangle$ can be seen as a
special case of the saturation property.  Indeed, the decreasing
nested sequence of internal sets,~$\langle{}^\ast\!K_n\colon
n\in\N\rangle$, has a common point~$x$ by saturation.  But for a
compact set~$K_n$, every point of~${}^\ast\!K_n$ is nearstandard
(i.e., infinitely close to a point of~$K_n$).  In particular,
$\sh(x)\in K_n$ for all~$n$, as required.

Hamkins wrote in \cite{Ha15b}: ``To give a sample consequence of
saturation, we observe that every infinite graph, no matter how large,
arises as an induced subgraph of a hyperfinite graph in any
sufficiently saturated model of nonstandard analysis.  This often
allows one to undertake finitary constructions with infinite graphs,
modulo the move to a nonstandard context.''

``The ultrapower construction can be extended to the power set of~$\R$
and its higher iterates.%
\footnote{See Section~\ref{s3}.}
In the end, one realizes that one might as well take the ultrapower of
the entire set-theoretic universe~$V$. One then has a copy of the
standard universe~$V$ inside the nonstandard realm~$\astv$, which one
analyzes and understands by means of the ultrapower construction
itself.%
\footnote{However, this idea runs into technical issues.  The standard
universe~$V$ is not a set, and the ultrapower of the membership
relation in~$V$ is not well-founded.  We discuss the ways in which
nonstandard analysis deals with these issues in Section~\ref{s4}.}
A few applications of nonstandard analysis exploit not just a single
ultrapower, but an ultrapower construction repeated along some linear
order.  Such iterated ultrapower constructions give rise to many
levels of nonstandardness, a useful feature.  Ultimately one is led to
adopt all of model theory as one's toolkit.'' \cite{Ha15b}

\section{Superstructure perspective}
\label{s3}
As Hamkins points out, before long, one wishes nonstandard analogues
of the power set~$\mathcal{P}(\R)$ and its higher iterates.  We will
now implement this idea.

\subsection{Ultrapower of the power set of~$\R$}
\label{s35}

Elements of the ultrapower of~$\mathcal{P}(\R)$ are the equivalence
classes of sequences~$\langle A_n \colon n\in \N \rangle$ of
subsets~$A_n\subseteq \R$ where sequences~$\langle A_n\rangle$
and~$\langle B_n \rangle$ are defined to be equivalent if and only if
we have~$\{ n \in \N \colon A_n=B_n \}\in U_{\xi}$.

The relation~$\astin$ between~$x =[\langle x_n\rangle]$ in~$\astr$
and~$[\langle A_n\rangle]$ in the ultrapower of~$\mathcal{P}(\R)$ is
defined by setting
\[
x\;\astin [\langle A_n\rangle] \text{\quad if and only if\quad} \{
n\in\N \colon x_n\in A_n\}\in U_{\xi}.
\]

In this construction, the sets in the ultrapower of~$\mathcal{P}(\R)$
are not subsets of\,~$\astr$ (they are equivalence classes of
sequences), and the membership relation~$\astin$ is not the usual
membership relation~$\in$. Both of these problems are solved by the
following stratagem.  With each equivalence class~$[\langle A_n\colon
n\in \N \rangle]$ in the ultrapower of~$\mathcal{P}(\R)$ we associate
a subset~$A$ of~$\astr$ as follows:
\[ x \in A \text{\quad if and only if\quad} x \; \astin [\langle
A_n\rangle].
\]
The subsets of~$\astr$ associated with the members of the ultrapower
of~$\mathcal{P}(\R)$ in this way are called \emph{internal sets}.  The
collection of all internal subsets of~$\astr$ is denoted
${}^{\ast}\mathcal{P}(\R)$.  We have~${}^{\ast}\mathcal{P}(\R)
\subsetneqq \mathcal{P}(\astr)$. The inclusion is strict because there
are subsets of~$\astr$ that are not internal; for example, the subset
$\R\subseteq\astr$ is not internal.

In the special case when~$A_n=S\subseteq\R$ for all~$n\in\N$, the
associated internal set is the natural extension of~$S$ introduced
earlier and denoted~$\asts$.  We shall also refer to~$\asts$ as the
\emph{standard copy} of~$S$.

\begin{definition}
\label{s42}
A \emph{hyperfinite set}~$X$ is an internal set associated with the
equivalence class of a sequence~$\langle A_n\colon
n\in\mathbb{N}\rangle$ where each~$A_n$ is finite.
\end{definition}

Such a set admits an internal enumeration by a hypernatural number,
denoted~$|X|$.  In more detail, the hypernatural number~$|X|$ is the
coset of the sequence of integers~$|A_n|$ where~$|A_n|$ is the
ordinary number of elements in the finite set~$A_n$.

It can be shown that every subset of~$\astr$ definable in
$\astr\cup{}^{\ast}\mathcal{P}(\R)$ from internal parameters is
internal, a fact known as the \emph{internal definition principle}.
Thus, if one defines a new object by a formula exploiting only
internal objects, the new object is necessarily internal, as well.
For instance, if one defines~$T$ to be the set of integers between~$1$
and an infinite hyperinteger~$H$, then~$T$ is necessarily internal.
Thus one needn't specify a presentation of the internal set~$T$ with
respect to the ultrapower construction.

The Transfer Principle now applies to formulas that quantify over both
\emph{elements} and \emph{sets of elements} of~$\R$.  Such statements
transfer into statements true about all \emph{internal} sets, but not
necessarily about all sets.  Thus for example every nonempty internal
set of hyperreals bounded above has a supremum in~$\astr$, but not
every set of hyperreals does.%
\footnote{The subset~$\R \subseteq \astr$ is a counterexample: it is
bounded above by every positive infinitely large number~$L$, but it
does not have a least upper bound: if~$L$ is an upper bound for~$\R$,
then~$L-1$ is similarly an upper bound.}

This construction can easily be extended to the second and higher
iterates of the power set operation. But the few examples already
given indicate that, for the practice of analysis, it suffices to know
that an extension of the standard structure exists, with suitable
properties, such as the Transfer Principle; the actual construction is
of secondary importance.

The superstructure framework enables one to step back from
the details of the actual (ultrapower) construction.%
\footnote{Such a perspective is comparable to the way mathematicians
think of real numbers.  They certainly wish to know that real numbers
can be conceived as equivalence classes of Cauchy sequences of
rationals.  However, he also wishes to be able to ignore the actual
construction of~$\mathbb{R}$ entirely in his everyday work, and to
consider real numbers as individuals (atomic entities) and~$\R$ simply
as a complete ordered field.}
It is still the most popular vehicle for the practice of nonstandard
analysis; see for example Robinson--Zakon \cite{RZ}, Chang--Keisler
\cite{CK}, Albeverio et al.\;\cite{Al86}.  We describe it in the next
subsection.

\subsection{The superstructure framework}
\label{s51}

For any set of individuals~$X$, the \emph{superstructure over}~$X$ is
obtained from ~$X$ by taking the power set countably many times.  In
more detail, one defines the superstructure recursively by
setting~$V_0(X) = X$,~$V_{n + 1}(X) = V_n(X)\cup
\mathcal{P}(V_{n}(X))$ for~$n \in \omega$, and~$V_{\omega}(X) =
\bigcup_{n < \omega} V_{n}(X)$.

The superstructure framework for nonstandard analysis consists of two
superstructures,~$V_{\omega}(X)$ and ~$V_{\omega}(Y)$ where~$X
\subsetneqq Y$ are infinite sets, and a mapping~$\ast\colon
V_{\omega}(X) \to V_{\omega}(Y)$ such that~${}^{\ast}\! a =a~$ for
all~$a \in X$,~$\astx= Y$, and~$A \in B$ implies~$\astA\in\astb$
for~$A, B \in V_{\omega}(X)$.  From now on, we take ~$X = \mathbb{R}$.
This facilitates comparison with the construction perspective;
meanwhile the discussion in this section applies fully to the general
situation.  Note that~$V_{\omega}(\mathbb{R})$ contains all objects of
interest to classical analysis, such as the fields of real and complex
numbers, the higher-dimensional spaces~$\R^n$ and~$\mathbb{C}^n$,
functions on these, collections of functions, Hilbert spaces~$\ell^2$
and~$L^2$, functionals, etc.

We shall refer to~$V_{\omega} (\mathbb{R})$ as the \emph{standard
universe} and to ~$V_{\omega} (\astr)$ as the \emph{nonstandard
universe}.  Thus, every set~$S$ in the standard universe~$V_{\omega}
(\mathbb{R})$ has a \emph{standard copy}~${}^{\ast}\!S$ in the
nonstandard universe (cf.\;Sections~\ref{s24} and~\ref{s35}).  The
sets contained in~$V_{\omega} (\astr)$ that are elements of
some~$\asts$ for~$S \in V_{\omega} (\mathbb{R})$ are called
\emph{internal}, and~$\astv_{\omega} (\mathbb{R})$ denotes the
collection containing the elements of~$\astr$ together with all
internal sets; we refer to it as the \emph{internal universe}.  Note
that~$\astv_{\omega} (\mathbb{R}) = \bigcup_{n < \omega}
\astv_{n}(\mathbb{R})$.  We have a proper inclusion~$\astv_{\omega}
(\mathbb{R}) \subsetneqq V_{\omega} (\astr)$; for example,~$\mathbb{R}
\subseteq \astr$ is not an internal set.  Obviously,~$V_{\omega}
(\mathbb{R}) \subseteq V_{\omega} (\astr)$, but all sets
in~$V_{\omega} (\mathbb{R})$, except for the hereditarily finite ones,
are \emph{external} (not internal).

The main principle that connects the two superstructures is the
\emph{Transfer Principle}, which posits that any property expressible
in the language of the superstructures by a bounded quantifier
formula%
\footnote{A \emph{bounded quantifier formula} is a formula where all
quantifiers have the form~$(\exists x \in y)$ or~$(\forall x \in y)$.}
holds in the standard universe~$V_{\omega}(\mathbb{R})$ about some
objects if and only if it holds for the standard copies of those
objects in the internal universe~$\astv_{\omega} (\R)$ or,
equivalently, in the nonstandard universe $V_\omega (\astr)$.

%

By the Transfer Principle,~$\langle\astr, {}^{\ast}\!+,
{}^{\ast}\!\times, {}^{\ast}\!\!\!<, 0,1\rangle$ is a field of
hyperreals with the properties as in Section~\ref{s24}.  In
particular, ~$\astr$ contains nonstandard numbers, such as
infinitesimals and infinitely large numbers.

Many arguments in nonstandard analysis rely on a \emph{saturation
principle} (see Definition~\ref{f11}).  Two useful forms of saturation
available in the superstructure framework are the following:
\begin{enumerate}
\item 
If~$A_1 \supseteq A_2 \supseteq \ldots{}$ are nonempty internal sets,
then~$\displaystyle\bigcap_{n < \omega} A_n \neq \varnothing$.
\item
If~$\{A_i\colon i \in I\}$ is a collection of sets in~$V_{\omega}
(\mathbb{R})$ possessing the finite intersection property, then
$\displaystyle\bigcap_{i \in I} \astA_i \neq \varnothing$.
\end{enumerate}

We shall not give any details of the construction of superstructure
frameworks. The advantage of this approach lies precisely in the fact
that there is no need to do so. One can accept the framework and
proceed to use it in any area of mathematics, while leaving its
construction to specialists.  Such constructions proceed along the
lines of Section~\ref{s35}; see \cite{CK}.  The internal
universe~$\astv_{\omega} (\mathbb{R})$ is isomorphic to the union of
the ultrapowers of~$V_{n} (\mathbb{R})$ for all~$n \in \omega$;%
\footnote{More precisely, to the bounded ultrapower of
$V_{\omega}(\R)$.}
the \emph{nonstandard universe}~$V_{\omega}(\astr)$ contains
additional, external sets.

\subsection
{An application: Loeb's construction of Lebesgue measure}

As an example of the interplay between the various kinds of sets in
the superstructure, we describe a nonstandard construction of the
Lebesgue measure on~$[0,1]$ due to P.\;Loeb \cite{Lo75} (with an
improvement that seems to have appeared first in
Hrbacek\;\cite{Hr79}).

We fix an infinite integer~$N \in \astn$.  Let~$t_i=\frac{i}{N}$
for~$i=0,\ldots, N$.  The hyperfinite set~$T=\{t_0,t_1,\ldots,t_N\}$
(see the sentence following Definition~\ref{s42}) is referred to as
``hyperfinite time line" in Albeverio et al.\;\cite{Al86}.  Let~$\mu$
be the counting measure on~$T$, i.e.,~$\mu(X)= |X|/|T|$ for each
internal set~$X \subseteq T$; then~$\mu$ itself is also internal.  For
every~$A \subseteq [0,1]$, the set
\[
\mathbf{sh}^{-1}[A] = \{z\in{}^{\ast}\hspace{-0.5pt}[0,1]
\colon\mathbf{sh}(z)\in A\}
\]
is the, generally external, set of
all~$z\in{}^{\ast}\hspace{-0.5pt}[0,1]$ such that~$z \approx x$ for
some~$x\in A$.  Finally, we define
\[
m_L(A)= \mathbf{inf} \Bigl\{ \mathbf{sh}(\mu(X))\colon X \subseteq T
\text{ is internal and }\mathbf{sh}^{-1}[A] \cap T \subseteq X
\Bigr\}.
 \]
The collection in braces is a (nonempty, bounded below) subset
of~$\R$, so the infimum exists.  It can be shown (see~\cite{Al86} for
some details) that~$m_L$ is the Lebesgue outer measure on~$[0,1]$; in
particular, a set~$A \subseteq [0,1]$ is Lebesgue measurable if and
only if one has~$m_L(A) + m_L([0,1] \smallsetminus A) = 1$, and in
this case~$m_L(A)$ is precisely the Lebesgue measure of~$A$.

\subsection{Axiomatizing}

It is often thought worthwhile (for example, for pedagogical reasons;
see \cite{17h}) to develop the subject purely from general principles
that make the nonstandard arguments succeed.  This approach is similar
to, say, deriving results from the axioms for algebraically closed
fields rather than arguing about these mathematical structures
directly.  As an example of such an exposition, Keisler's calculus
textbook is based on five simple axioms A through E that hold in every
superstructure framework; see Keisler \cite{Ke86} and \cite{Ke76}.
 
 \bigskip
 The next step would be to take the entire universe of
sets,~$V$, as a ``superstructure."  We already indicated that this
move gives rise to some issues, both technical and fundamental.%
\footnote{While the ultrapower of~$V_n (\mathbb{R})$ for~$n \in
\omega$ is well-founded, and hence isomorphic to a transitive set, by
the Mostowski collapse lemma, when one takes the ultrapower
of~$V_{\omega}(\mathbb{R})$, or even of the entire~$V$, the membership
relation~$\astin$ is non-well-founded, and therefore the ultrapower is
not isomorphic to a transitive set or class in the framework of ZFC.
This situation calls for an ontological commitment beyond what is
required by ZFC, and is best handled axiomatically.}
 It is considered in the next section.

\section{Axiomatic perspective}
\label{s4}

\subsection{Nonstandard set theories}
\label{s51b}

As mentioned above, taking an ultrapower of the entire set-theoretic
universe~$V$ is an attractive idea.  Since~$V$ is not a set in ZFC,
the best way to handle this construction is axiomatic. The earliest
axiomatizations of ultrapowers of $V$ were proposed by Petr Vop\v enka
in \cite{VH} as part of his project of axiomatizing the method of
forcing.  The first nonstandard set theories that extend ZFC and can
serve as a framework for the practice of nonstandard analysis were
developed independently by Hrbacek \cite{Hr78} and Nelson
\cite{Ne77}.%
\footnote{Vop\v enka also developed a nonstandard set theory that
enables infinitesimal methods. His Alternative Set Theory (AST) is
incompatible with ZFC and is not considered here.  See Sochor
\cite{So76} and Vop\v enka \cite{Vo79}.}
We refer to Kanovei--Reeken \cite{KR} for a comprehensive survey of
the field.

Nonstandard set theories typically possess the set membership
predicate~$\in$ and a unary predicate of standardness~$\st$;
here~$\st(x)$ reads \mbox{``$x$ is standard.''}  The analog of
$\st(x)$ in the superstructure framework is ``$x$ is a
standard copy."

A helpful classification tool is a distinction between \emph{internal}
and \emph{external} set theories.
 
\emph{Internal} set theories axiomatize only the standard and internal
sets.  Edward Nelson's IST is the best known example.  The theory IST
postulates all the axioms of ZFC in the pure~$\in$-language, together
with three extra axioms (more precisely, axiom schemata) of
Idealization, Standardization, and Transfer, which govern the
relations between standard and nonstandard sets.  The theory is
particularly attractive because of its simplicity.  Kanovei
\cite{Ka91} modified IST by limiting its universe to bounded sets.%
\footnote{A set is \emph{bounded} if it is an element of a standard set.}
The resulting internal theory BST%
\footnote{BST was explicitly formulated in \cite{Ka91}.  Implicitly it
is equivalent to the theory of internal sets in some of the theories
developed in Hrbacek \cite{Hr78}; see Section~\ref{s53}.}
provides the same mathematical tools as IST, but it has better
metamathematical behavior; see Section~\ref{s52}.

\subsection{The axioms of BST}
\label{s52b}

We shall now formulate the axioms of BST.  The
notation~$(\exists^{\st} x)\ldots{}$
and~$(\forall^{\st}x)\ldots{}$ abbreviates respectively the
formulas~$(\exists x)(\st(x)\wedge\ldots)$ and~$(\forall
x)(\st(x)\Rightarrow\ldots)$.  The
notation~$(\forall^{\st\textbf{fin}} a)\ldots{}$ is shorthand
for the formula
\[
(\forall a)\left((\st(a)\wedge (a\text{\;is
finite}))\Rightarrow\ldots_{\phantom{I}} ^{\phantom{I}} \right).
\]
The axioms of BST include those of ZFC (with \emph{Separation} and
\emph{Replacement} only for formulas in the~$\in$-language) together
with the following four additional principles.\\

\subsection*{Boundedness}
~\\~$(\forall x)(\exists^{\st} y ) (x \in y).$

This means that every set is an element of \emph{some} standard set.
This is the main difference between BST and IST.\, In IST there is a
set that contains \emph{all} standard sets as elements.\\

\subsection*{Bounded Idealization}
~\\~$(\forall^{\st} A ) \left[ (\forall^{\st{}\textbf{fin}} a \subseteq
A ) (\exists y)( \forall x \in a ) \Phi (x, y) \equiv (\exists
y)(\forall^{\st} x \in A ) \Phi (x,y) \right]$\\ where~$\Phi(x, y)$ is
any formula in the~$\in$-language, possibly with parameters.

This axiom is a version of saturation or compactness; see
Definition~\ref{f11}.  Loosely speaking, if a collection of
properties~$\Phi(x,y)$,~$x\in A$ is finitely satisfiable, that is, for
every standard finite set~$a = \{x_1,\ldots, x_n\}$ of elements of~$A$
there is some~$y$ such that~$\Phi(x_1,y) \wedge \ldots \wedge
\Phi(x_n,y)$ holds, then there is some~$y$ that satisfies~$\Phi(x,y)$
simultaneously for all standard~$x \in A$.  As a basic example,
consider the property
\[
(i,j \in \N)\wedge(i<j).
\]
For every standard finite set~$a = \{i_1,\ldots, i_n\} \subseteq \N$
there is~$j \in \N$ such that~$(i_1<j)\wedge\ldots\wedge(i_n<j)$
(namely, take~$j=\max\{i_1,\ldots,i_n\} +1$).  Hence by \emph{Bounded
Idealization} there exists an element~$j \in \N$ such that~$i < j$ for
all standard~$i$; i.e.,~$j$ is an infinitely large natural number.\\

\subsection*{Transfer}
~\\~$(\exists x) \Phi (x) \equiv (\exists^{\st} x) \Phi (x)$ \\
where~$\Phi(x)$ is any formula in the~$\in$-language, with standard
parameters.

An equivalent version of \emph{Transfer} is that any statement
$\Phi(a_1,\ldots, a_n)$ (in the~$\in$-language) about standard sets
$a_1,\ldots, a_n$ is true when interpreted in the standard universe if
and only if it is true when interpreted in the internal universe. This
is a modern version of Leibniz's \emph{Law of Continuity}; see
Section~\ref{s31}.\\

\subsection*{Standardization}
For each formula~$\Phi(z)$ in the~$(\st,\in)$-language, possibly with
parameters, we have the following:~\\~$(\forall^{\st} X )
(\exists^{\st} Y ) (\forall^{\st} z ) (z \in Y \equiv z \in X \wedge
\Phi(z))$

No axiom of \emph{Separation} for arbitrary formulas in
the~$(\st,\in)$-language is imposed by BST; for example, the
collection~$\{n\in\N\colon\st(n)\}$ is not a set.%
\footnote{More precisely, BST (as well as IST) proves that there is no
set~$x$ equal to~$\N'=\{n\in\N\colon\st(n)\}$.  Entities like~$\N'$,
that is, those defined as~$\{x \in X\colon \Phi(x)\}$, where~$X$ is a
true set and~$\Phi$ is a formula in the~$(\in,\st)$-language, are
called \emph{external sets} (or sometimes \emph{semisets}), and
typically they are non-sets in BST and IST, unless \st{} can be
eliminated in some fashion.

Yet BST allows for an implicit introduction of external sets by means
of a certain coding defined in \cite{KR95}.  This coding system
involves not only external sets of internal elements like the
collection of all standard integers~$\N'$ above, but also external
sets of external sets, etc.  This results in the construction of a
universe of Hrbacek's external set theory HST, which extends the given
BST universe in the same way that the complex numbers extend the real
line by means of representation of a number~$a+bi$ as a pair~$\langle
a,b\rangle$ of real numbers~$a,b$. See \cite{KR95} or
\cite[Chapter\;5]{KR}.}
In its place one has the \emph{Standardization} principle, to the
effect that for any property~$\Phi(x)$ expressible in
the~$(\st,\in)$-language and any standard set~$X$ there is a standard
set~$Y$ that contains exactly those standard elements of~$X$ that have
the property~$\Phi$. (But~$Y$ and~$\Phi$ do not have to agree on
nonstandard elements of~$X$.) Two simple but important consequences of
the BST axioms, especially of \emph{Standardization}, are the
following:
\begin{enumerate}
\item
If~$k,n \in \N$,~$n$ is standard, and~$k<n$, then~$k$ is also
standard; and
\item
If~$x~$ is a finite real number (i.e.,~$|x| < n$ for some
standard~$n\in \N$), then there is a (unique) standard real number~$r$
such that
\hbox{$r \simeq x$} (such an~$r$ is called the \emph{standard part}
of~$x$ or the \emph{shadow} of~$x$).
\end{enumerate}

We first prove (2).  Assume~$|x| < n$ for some standard~$n \in
\mathbb{N}$.  By \emph{Standardization}, there is a standard set~$Y~$
such that
\[
(\forall^{\mathbf{st}} z)(z \in Y \equiv z\in\mathbb{R}\:\wedge\:z<x).
\] 
Note that~$Y \neq \emptyset$ (since~$-n \in Y$) and~$Y$ is bounded
above by~$n$. By completeness of~$\mathbb{R}$,~$Y$ has a supremum~$r$,
and by \emph{Transfer},~$r~$ is standard.

We now show that~$x \simeq r$, so~$r$ is the standard part of~$x$.  If
not, then~$|r-x| >s >0$ for some standard~$s$.  This means that either
$x > r+s$ or~$x< r-s$.  In the first case,~$r+s \in Y$, contradicting
$r=\sup Y$.  In the second case,~$r-s~$ is an upper bound on~$Y$,
again contradicting~$r=\sup Y$.

Now we show that (1) follows from (2).  Since~$k$ is assumed to be
finite, it has a standard part~$r$. Then~$k$ is the unique integer in
the standard interval~$[r - 0.5, r + 0.5)$, so~$k$ is standard by
\emph{Transfer}.  ~\\

We note that \emph{Boundedness} is a single axiom while \emph{Bounded
Idealization}, \emph{Standardization}, and \emph{Transfer} are
\emph{axiom schemata}, that is, they apply to an arbitrary
formula~$\Phi$ (of a certain type); in this they resemble the schemata
of \emph{Separation} and \emph{Replacement} of ZFC.

The schemata of \emph{Standardization} and \emph{Transfer} are common
with IST.  The schema of \emph{Bounded Idealization} is weaker than
the full \emph{Idealization} of IST, but the \emph{Boundedness} axiom
makes up for it, adding a more comprehensive control over the
interactions between standard and internal sets in the BST set
universe than it is possible in IST.

\subsection{Connection with ultrafilters}

Let us clarify the connection of these axioms with ultrafilters.
Working in BST and given an internal set~$x$ and a standard
ultrafilter~$U$, we say that
\begin{quote}
$x$ \emph{is in the monad of\; }~$U$ \;if\;~$x \in A$ for all
standard~$A \in U$.
\end{quote}
Note that, for~$x$ in the monad of~$U$,~$x$ is nonstandard if and only
if~$U$ is nonprincipal.

From \emph{Boundedness} and \emph{Standardization} it follows that for
every~$x$ there is a standard ultrafilter~$U$ such that~$x$ is in the
monad of~$U$.

From \emph{Bounded Idealization} it follows that for every standard
ultrafilter~$U$ there is an~$x$ such that~$x$ is in the monad
of~$U$. This is a version of saturation.

In fact, for somewhat stronger versions of these statements (the Back
and Forth properties, Hrbacek\;\cite[Section 5]{Hr09}), the
implications can be replaced by equivalences.

As a consequence, the axiom schemata of BST can actually be replaced
by single axioms.  Thus BST is finitely axiomatizable over ZFC; see
Kanovei--Reeken \cite[Section 3.2]{KR}.

The connection goes deeper.  Assuming that~$x$ is in the monad of~$U$,
where~$U$ is a standard ultrafilter on a standard set~$I$, we have a
mapping that assigns to (the equivalence class modulo~$U$ of) each
standard function~$f$ on~$I$ the internal set~$f(x)$. This mapping is
an isomorphism between the standard ultrapower of~$V$ modulo~$U$, with
the~${}^{\ast}\!\!\!\in$-relation, and the class~$\{ f(x) \colon f
\text{ is standard}\, \}$ of internal sets, with the~$\in$-relation
\cite[Section~6.1]{KR}.

\subsection{External theories}
\label{s53}

Two theories proposed in Hrbacek \cite{Hr78} and assigned acronyms HST
and NST in Kanovei--Reeken \cite{KR} are \emph{external} set theories.
They are formulated in the~$(\inte, \st, \in)$-language, where~$\inte$
is a unary predicate of internality (here\;$\inte(x)$ reads ``$x$ is
internal'').  The internal part of these theories satisfies BST, but
they admit also external sets.  This makes them more powerful, but
also more complex.

The two theories differ in the properties of the universe of all
(internal and external) sets.  It turns out that this universe cannot
satisfy all the axioms of ZFC.  Thus, HST allows \emph{Replacement}
(and even its stronger version, \emph{Collection}), while NST makes
available \emph{Power Set} and \emph{Choice}.

The theory KST proposed in Kawai \cite{Ka81} extends IST by external
sets.  Andreev and Gordon developed the nonstandard class theory NCT
\cite{AG}.  Roughly speaking, NCT is to BST what the von
Neumann--G\"{o}del--Bernays set theory is to ZFC; in particular, NCT
has standard, internal and external classes.

The work of Nelson and his followers demonstrated that a lot of
nonstandard mathematics can be carried out by internal means alone;
see Section~\ref{s71s}.  External sets are necessary for some more
advanced constructions, such as Loeb measures.

There are also nonstandard set theories that do not fit into our
two-way classification of axiomatic approaches that well; we mention
only the~$\alpha$-theory proposed in Benci--Di Nasso\;\cite{BD} and
the framework for nostandard analysis developed in ZFC with the axiom
of \emph{Foundation} replaced by the axiom of
\emph{Superuniversality}, by Ballard and Hrbacek\;\cite{BH92}.

\subsection
{Consistency, conservativity and intuitive interpretation}
\label{s52}

Each of these theories is known to be an equiconsistent and
conservative extension of ZFC.  Here \emph{conservativity} means that
any~$\in$-sentence provable in the nonstandard theory is already
provable in ZFC; the converse is clearly true as well.

Therefore one might expect that each model of ZFC can be embedded, as
the class of all standard sets, into a model of the nonstandard
theory.  However, this is false for IST; see Kanovei--Reeken
\cite[4.5]{KR}.  In fact ZFC has to be strengthened, for instance by
adding the global choice axiom and the truth predicate
for~$\in$-formulas, in order to be able to embed its model, as the
class of standard sets, into a model of IST; see \cite[4.6]{KR}.  For
NST this expectation is also false, but it is true for BST as well as
for HST.

We pointed out already that an ultrapower of~$V$ is non-well-founded,
and hence in ZFC there is no mapping of the ultrapower of~$V$ onto
some transitive class that would convert the membership relation
$\astin$ in the ultrapower into a true~$\in$-relation. Briefly, the
membership relation in the ultrapower is not the actual membership
relation~$\in$, and cannot be converted to it in the framework of
ZFC.\, Since nonstandard set theories axiomatize ultrapowers or
iterated ultrapowers, they transcend ZFC; an intuitive picture of the
universe of such a theory properly extends the familiar ZFC
universe~$V$.

A question then arises to determine what the place of~$V$ is in this
larger picture.  In a discussion of this issue, we have to distinguish
between two meanings of the word \emph{standard}.  On one hand, there
is the technical meaning; standard objects are the objects in the
scope of the predicate~$\mathbf{st}$.  On the other hand, the
expression \emph{standard objects} is sometimes used to refer to the
objects that cousin Georg and many traditional mathematicians are
familiar with as the objects from the ZFC universe~$V$.  These two
meanings do not necessarily coincide; in order to perform a
disambiguation, we will refer to the \emph{standard objects} in this
second sense as the \emph{familiar objects}.

\subsection{The three pictures}
\label{s55}

At this point, there are several choices.  One can identify the
\emph{internal} numbers and sets with the familiar numbers and sets.
Standard reals and sets (i.e., those in the scope of~$\mathbf{st}$)
are on this view \emph{only some} of the mathematical entities with
which we are familiar, singled out for special attention. One can call
this view the \emph{internal picture}. It is the view adopted in
Nelson \cite{Ne77}.  We discuss it in detail in Section~\ref{s72}.

However, one can equally well regard the \emph{standard} numbers and
sets as the familiar numbers and sets, with the understanding that
standard sets may contain also (what are referred today as)
nonstandard elements, or (as Leibniz might have called them)
\emph{ideal elements} or \emph{useful fictions}; see
Section~\ref{s85}.  One can call this view the \emph{standard
picture}.  It is proposed in Hrbacek\;\cite{Hr79}.  In the
superstructure framework, it would correspond to viewing the
\emph{standard copies} as the familiar objects.  Analogously, in
projective geometry an affine plane is outfitted with an assortment of
ideal points at infinity; see Section~\ref{s31}.

External set theories admit yet another picture, the \emph{external
picture}.  In these theories one can single out the class of
well-founded sets.%
\footnote{A set~$A$ is \emph{transitive} if any~$x \in A$
satisfies~$x\subseteq A$.  A set~$X$ is \emph{well-founded} if there
is a transitive set $A$ such that~$X \subseteq A$ and the restriction
of the $\in$-relation to~$A$ is well-founded. In ZFC all sets are
well-founded, but external nonstandard set theories necessarily have
also ill-founded sets.}
These are generally external.  There is a one-one~$\in$-preserving
mapping~$\ast$ of the well-founded sets onto the standard sets.  One
then has the option of regarding the well-founded sets as the
intuitively familiar objects, and everything else as \emph{ideal}. In
this view, mathematics can be developed in a way similar to the
superstructure framework, with the universe of all well-founded sets
in place of~$V_{\omega}(\R)$.  This picture was outlined in an
appendix to Hrbacek\;\cite{Hr79} and is fully implemented in
Kanovei--Reeken \cite{KR}.

\subsection{The Protozoa metaphor}

The following analogy may be helpful in interpreting the three
pictures of Section~\ref{s55}.  Here we are thinking of each real
number as an individual (as one does in the superstructure approach).
Let us assume we are familiar with the class of animals, and then
someone invents the microscope and we discover the protozoa; are we to
count them as animals?  There are three possible answers.
\begin{enumerate}
\item
(External picture) Protozoa are not animals because they fall outside
what we previously meant by the word `animal', but we can invent a new
word, `hyperanimal,' to include animals and protozoa.
\item
(Standard picture) Protozoa are animals and were so all along; there
is no change in the meaning of `animal'.  The familiar class of
animals contains unfamiliar species like protozoa in addition to the
familiar ones.  The microscope allows us to see new animals (animals
that are new to us, that is).
\item
(Internal picture) Protozoa are animals, and they were familiar all
along, even though we weren't aware of them specifically.  What is new
is our ability to distinguish between microscopic and macroscopic
animals: this is what the microscope provides.%
\footnote{The allusion to Keisler's microscope Keisler \cite{Ke86} is
of course intentional.}
\end{enumerate}

\subsection{The intuitive interpretation}
 
The following three points need to be kept in mind.

1. The choice of the nonstandard set theory does not commit one to a
particular \emph{picture}.  Nelson used the internal picture, and
hence this picture is usually considered an intrinsic part of IST, but
IST is equally compatible with the standard picture. The theory and
the picture are two separate things. The choice of the picture is not
a mathematical issue, since one will get exactly the same mathematical
results regardless of the picture, but rather a matter of personal
preference.

2.  Cousin Georg would likely object to both the internal and the
standard picture, on the grounds that numbers and sets as developed
around 1872 not long before the 1886 Patent-Motorwagen (see
Section~\ref{wings}), i.e., the entities he feels comfortable with, do
not fit either in his opinion.

In the case of the internal picture the issue is that, intuitively,
every nonempty collection of natural numbers has a least element.  Now
in a nonstandard set theory like IST there is no \emph{set} of all
nonstandard natural numbers.  This may run counter to cousin Georg's
expectations concerning the properties of the \emph{familiar} numbers
along the lines of the CD+II mindset (see Section~\ref{s11}), where
every definite property of natural numbers should necessarily
determine a set.  Note that~$\mathbf{st}$ is treated as a definite
property in nonstandard set theory: for each~$x$,
either~$\mathbf{st}(x)$ or~$\neg \mathbf{st}(x)$.

In \emph{intuitive} terms, we can consider~$\st$ to be an indefinite,
vague property like \emph{heap} in the \emph{paradox of sorites}.
This indefiniteness is reflected in the formal theory by the
nonexistence of the corresponding set.  The resistance to this idea is
due to an inherent tendency to abstract (form collections).

Cousin Georg's objection to the standard picture is that the
\emph{familiar} set~$\mathbb{N}$ just does not contain any ideal
elements.  A counter-argument is that one can think \emph{as if} it
did.  The problems connected with collections of \emph{ideal} natural
numbers without the least element are also perhaps less pressing in
this picture.  (A personal anecdote: The recent book
Hrbacek--Lessman--O'Donovan \cite{HLO} was first written in the
internal picture, which was strongly objected to by some of the
referees--perhaps cousin Georg among them.  Eventually a switch to the
standard picture was implemented, and the book was accepted
immediately afterward.)

3.  The external picture seems to be the least objectionable
from the point of view of cousin Georg (in fact, there seems to be no
obvious reason why it should be objectionable at all), but the
necessity to work with external sets (in addition to the standard and
internal ones) complicates the framework.

A discussion of these various pictures can be found in Hrbacek
\cite{Hr06}.

\subsection{Relative standardness}
\label{s57}

Arguably, there is not much current mathematical work in nonstandard
analysis that actually uses more than one level of standardness. In
the superstructure framework, we know only of an early paper Molchanov
\cite{Mo89}, employing two levels.  Perhaps one reason for this
paucity is that the model-theoretic framework with more than one level
of standardness quickly becomes unmanageable because it naturally
involves complex combinations of ultrapowers to take care of all the
details.

As for the axiomatic framework, there are two distinct approaches. 

The first approach is to define different levels or degrees of
standardness within a given nonstandard universe of discourse, e.g., a
universe satisfying Nelson's internal set theory IST or its BST
(bounded set theory) version; see Section~\ref{s51b} and
Section~\ref{s52b}.

There are several meaningful approaches to defining relative
standardness in this setting.  Among them are the following three:
\begin{enumerate}
\item
(most natural but not most useful) A set~$y$ is standard with respect
to~$x$ if there is a standard map~$f$ such that~$x\in dom\, f$ and
$y=f(x)$; see e.g., Gordon \cite{Go89}.
\item
(following \cite{Go89}) A set~$y$ is standard with respect to~$x$ if
there is a standard map~$f$ such that~$x\in dom\, f$, all values
of~$f$ are finite sets%
\footnote{Finiteness in IST corresponds to hyperfiniteness in the
model-theoretic approach.}
and ~$y \in f(x)$. 
\item
(following Kanovei \cite[Section 3]{Ka91}, based on earlier ideas of
Luxemburg\;\cite{Lu62} and Hrbacek).  If~$\kappa$ is a standard
cardinal then one defines \emph{sets\;of\;order}~$\kappa$ as those
belonging to standard sets of cardinality~$\kappa$ or less.
\end{enumerate}
These and similar definitions lead to the \emph{internal subuniverses}
as discussed in Kanovei--Reeken \cite{KR}, that is, classes~$I$ of
internal sets satisfying the following: if~$x_1,\ldots,x_n\in I$,
where~$n$ is a standard number,~$f$ is a standard function, and the
string~$\langle x_1,\ldots,x_n \rangle$ belongs to~$dom\, f$
then~$f(x_1,\ldots,x_n)\in I$.  Each such an internal subuniverse~$I$
can be considered as a new degree of standardness, that is, the class
of all sets \emph{standard in the new sense}.

Chapter 6 in \cite{KR} is devoted to these notions and contains
several results related to properties of these generalized notions of
standardness.

The second approach to relative standardness is to introduce this
complex notion axiomatically, rather than explicitly defining it in a
given nonstandard universe.  The advantage of this approach is that
all the axioms of BST can be relativized to each level of
standardness.

The idea to treat levels of standardness axiomatically was proposed by
Wallet (see P\'eraire and Wallet \cite{PW}) and fully developed by
P\'eraire \cite{Pe92}.  P\' eraire's axiomatic theory RIST is an
extension of IST to many levels of standardness. Besides~$\in$, its
language contains a binary \emph{relative standardness}
predicate~$\sqsubseteq$; one can read~$x \sqsubseteq y$ as ``$x$ is
standard relative to~$y$.''  Thus the class~$\{ x\colon x \sqsubseteq
y \}$ is the \emph{level of standardness} determined by~$y$. The
number 0 (or any other object definable without parameters) determines
the coarsest level of standardness, which can be identified with the
standard sets of IST.

In RIST one can have a finite sequence of numbers, for example
$\eta_0$,~$\eta_1$,~$\eta_2$, where~$\eta_0$ is infinitesimal and each
of the other terms is infinitesimal relative to the level of
standardness determined by the preceding one.  Tao writes:
\begin{quote}
Having this hierarchy of infinitesimals, each one of which is
guaranteed to be infinitesimally small compared to \emph{any} quantity
formed from the preceding ones, is quite useful: it lets one avoid
having to explicitly write a lot of epsilon-management phrases such as
``Let~$\eta_2$ be a small number (depending on~$\eta_0$ and~$\eta_1$)
to be chosen later'' and ``\ldots assuming~$\eta_2$ was chosen
sufficiently small depending on~$\eta_0$ and~$\eta_1$'', which are
very frequent in hard analysis literature, particularly for complex
arguments which involve more than one very small or very large
quantity.  Tao \cite[p.\;55]{Ta08}
\end{quote}

While it is straightforward to iterate levels of standardness
countably many times, or even along any a priori given linear
ordering, it is harder to produce satisfactory frameworks that have
sequences of levels of standardness of arbitrary hyperfinite length.
This is done in Hrbacek \cite{Hr04} and \cite{Hr09}.  In such a
framework (the theory GRIST) one can have a hyperfinite sequence of
natural numbers where each term is nonstandard relative to the
previous one.  Tao speculated that this feature might be useful in the
proof of Szemer\' edi's theorem, but the only use of it thus far seems
to be a characterization of higher-order differentiability in Hrbacek
\cite{Hr07}.  See also \cite{TV} for a use of the language of
nonstandard analysis in order to avoid a large number of iterative
arguments to manage a large hierarchy of parameters.

The axiomatic framework is used in several papers of P\'eraire (see
e.g., \cite{Pe93}), as well as Gordon \cite{Go97}, the more recent
work of Di Nasso \cite{Di15}, and in \cite{HLO} (these works use
infinitely many levels of standardness).%
\footnote{In the simplest case one has a nested chain of classes
$\N^{\st_0} \subseteq \N^{\st_1} \subseteq \N^{\st_2} \subseteq \cdots
\subseteq \N$, where~$\N^{\st_0} = \N^{\st}$ can be identified with
the standard natural numbers and~$\N$ is the set of all natural
numbers.  Thus the higher the~$n$, the less standard the integers
of~$\N^{\st_n}$ are.  In fact the ordering of the levels of
standardness does not have to be of type~$\omega$.  It can be more
complicated, have infinite descending chains, even be dense, but it
has to have a least element~$\mathbb{N}^{\st_0}$.}
The axiomatic framework hugely simplifies the presentation.

While RIST and GRIST are internal theories, there are also external
theories with many levels such as SNST by Fletcher \cite{Fl89} and EST
by Ballard \cite{Bal94}.

\subsection{Constructive aspects of Robinson's framework}
\label{s5}

The discussion in some of the earlier sections may have given an
impression that nonstandard analysis depends on ultraproducts or some
similar model-theoretic techniques, and therefore is essentially
non-effective.  Errett Bishop and Alain Connes%
\footnote{\label{f29}See the references in note\;\ref{f5}.}
have both claimed that nonstandard analysis is somehow fundamentally
non-constructive and non-effective. These claims are fundamentally
flawed and have been debunked as follows.

\subsubsection{Abstract analysis}

Methods of nonstandard analysis are often applied in areas that are
inherently non-constructive, such as general topology, the theory of
Banach spaces, or Loeb measures; we refer to them here as
\emph{abstract analysis}. The nonstandard settings that one needs for
abstract analysis are of course also non-constructive, and do imply
the existence of ultrafilters. For example, in IST$^{-}$, a theory
obtained from IST by deleting the Axiom of Choice, one can prove the
Boolean Prime Ideal Theorem (Every filter can be extended to an
ultrafilter).%
\footnote{It may be interesting to note that one cannot prove the full
Axiom of Choice; so even the nonstandard abstract analysis, to the
extent it is supported by IST$^{-}$, is actually more ``constructive''
than the standard abstract analysis carried out in ZFC.\; See Hrbacek
\cite{Hr12} and Albeverio et al.\;\cite[p.\;31]{Al86} for a related
discussion.}
However, the approach not involving nonstandard analysis (if it
exists) is \emph{equally non-constructive}.  A detailed example of
this phenomenon in the work of Connes may be found in \cite{13c}.

\subsubsection{Nonstandard Peano Arithmetic theories}

Classical mathematical results generally do not require the full power
of set theory, whether standard or nonstandard.  Let PA be (the
axiomatic theory of) the Peano arithmetic, which naturally axiomatizes
the natural numbers~$\N$, and, for each~$n$, let PA$_n$ be the
\emph{$n$-th Peano arithmetic}, a theory which naturally axiomatizes
the structure which contains the set of natural numbers~$\N$ along
with its consecutive power sets~$\PP(\N)$,~$\PP_2(\N)=\PP(\PP(\N))$,
\ldots,~$\PP_{n-1}(\N)$, so that PA itself is PA$_1$.

The ideas similar to those described for ZFC, can also be applied to
define a nonstandard version~${}^\ast\hspace{-1pt}\text{PA}_n$ of each
PA$_n$.  The relations among these theories (between standard and
nonstandard versions for the identical and different values of~$n$)
were thoroughfully studied in Henson--Keisler \cite{HK86} and some
related papers.  In particular, it was established that each
${}^\ast\hspace{-1pt}\text{PA}_n$ is comparable in terms of its
strength rather with the standard theory PA$_{n+1}$ than with its
direct standard base PA$_n$.

\subsubsection{Constructive nonstandard mathematics}

There is extensive work by Palmgren, Avigad \cite{Av05}, Martin-L\"of
\cite{Ma90b}, van den Berg et al.\;\cite{Va12}, on \emph{constructive}
nonstandard mathematics; see Palmgren \cite{Pa16} for a bibliography
of the early contributions.  These references introduce both syntactic
and semantic approaches to nonstandard analysis which are constructive
in the sense of Bishop's \emph{constructive analysis} and
Martin-L\"of's \emph{constructive\;type theory}, i.e., based on
intuitionistic logic.

\subsubsection{Effective content and reverse mathematics}

\emph{Classical} nonstandard analysis actually contains a lot of
\emph{easily accessible} effective content as follows.  Sanders
\cite{Sa15} establishes that one can algorithmically convert a proof
of a theorem in ``pure'' nonstandard analysis (i.e., formulated solely
with nonstandard axioms and nonstandard definitions of continuity,
compactness, differentiability, etc.) into a proof of the
constructive/effective version of the associated classical theorem.

This work is done in an axiomatic framework (some fragment of full
nonstandard set theory) and always produces effective (and even
constructive) results when \emph{Transfer} and \emph{Standardization}
are not used.

The use of the former gives rise to \emph{relative computability
results} in the spirit of \emph{Reverse Mathematics}, while the use of
the latter (or saturation) translates into results computable modulo
\emph{bar recursion} \cite{Ko08}.  Osswald and Sanders discuss the
constructive content of nonstandard analysis at length in \cite{OS},
and Sanders \cite{Sa17} discusses how these results undermine the
Bishop--Connes critique.

\subsubsection{Tennenbaum's theorem}

At first glance, even fragments of Ro\-binson's framework based on
\emph{arithmetic} may seem fundamentally non-constructive from the
viewpoint of Tennenbaum's theorem.  The theorem \emph{literally}
states that any nonstandard model of Peano Arithmetic is not
computable.  What this means is that for a nonstandard
model~$\mathcal{M}$ of Peano Arithmetic, the
operations~$+_\mathcal{M}$ and~$\times_\mathcal{M}$ cannot be
computably defined in terms of the operations~$+_\mathbb{N}$
and~$\times_\mathbb{N}$ of an intended model~$\mathbb{N}$ of Peano
Arithmetic.  Similar results exist for fragments; see Kaye
\cite[\S\,11.8]{Kay91}.

Now, while certain nonstandard models do require non-constructive
tools to build, \emph{models} are not part of Nelson's axiomatic
approach IST or its variant BST (see Section~\ref{s52b}).
Furthermore, IST and BST specifically disallow the formation of
external sets like ``the operation~$+$ restricted to the standard
numbers.''  Nelson called attention to this rule on the first page of
\cite{Ne77} introducing IST: ``We may not use external predicates to
define subsets.  We call the violation of this rule \emph{illegal set
formation}''  (emphasis in original).

Thus, one of the fundamental components of Tennenbaum's theorem,
namely the external set ``$+$ restricted to the standard naturals'' is
missing from the internal set theories IST and BST, as the latter
exclusively deal with internal sets.  Arguably, therefore,
Tennenbaum's theorem is merely an artifice of the model-theoretic
approach to nonstandard analysis.

The critique by Connes of Robinson's framework is based on similarly
flawed assumptions, namely that the models generally used in
Robinson's framework are fundamentally non-constructive and therefore
so is nonstandard analysis.  It is worth pondering the fact that
non-constructive mathematics is routinely used in physics (see e.g.,
the discussions of the Hawking--Penrose singularity theorem and the
Calabi--Yau manifolds in\;\cite{11a}, undecidability of the spectral
gap \cite{CPW}), without scholars jumping to the conclusion that
physical reality is somehow non-constructive.

\section{Physics: Radically elementary modeling}
\label{s7b}

\subsection{Tao on intricate results}

Tao wrote in \emph{Compactness and contradiction} as follows:
\begin{quote}
The non-standard proofs require a fair amount of general machinery to
set up, but conversely, once all the machinery is up and running, the
proofs become slightly shorter, and can exploit tools from (standard)
infinitary analysis, such as orthogonal projections in Hilbert spaces,
or the continuous-pure point decomposition of measures.  Tao
\cite[p.\;168]{Ta13}
\end{quote}
%
%
Here Tao is referring to the fact that \emph{if} one works from the
\emph{construction perspective} (as he does), then the entire
ultrapower construction counts as part of the general machinery.
Meanwhile,
\begin{quote}
\ldots for particularly intricate and deep results it can happen that
non-standard proofs end up being simpler overall than their standard
analogues, particularly if the non-standard proof is able to tap the
power of some existing mature body of infinitary mathematics (e.g.,
ergodic theory, measure theory, Hilbert space theory, or topological
group theory) which is difficult to directly access in the standard
formulation of the argument. \cite[p.\;169]{Ta13}
\end{quote}
Edward Nelson would have likely subscribed to Tao's view as expressed
above, but may have added that there is another side of the coin which
is what we shall refer to as \emph{Radically Elementary Modeling}.
This term alludes to Nelson's book \emph{Radically Elementary
Probability Theory} \cite{Ne87}.

\subsection{The physicist's vibrating string}
\label{s71s}

We are interested in developing a mathematical model of a vibrating
string.  First we will illusrate the viewpoint of a physicist by
providing some quotations.  In the celebrated \emph{Berkeley Physics
Course} (vol.\;3), Frank Crawford writes:

\begin{quote}
(Sec.\;2.1) If a system contains a very large number of moving parts,
and if these parts are distributed within a limited region of space,
the average distance between neighboring moving parts become very
small. As an approximation, one may wish to think of the number of
parts as becoming infinite and the distance between neighboring parts
as going to zero. One then says that the system behaves as if it were
``continuous''.  Implicit in this point of view is the assumption that
the motion of near neighbors is nearly the same.  Crawford
\cite[p.\;48]{Cr66}
\end{quote}
The next quotation deals with a uniform beaded string having~$N$ beads
and with fixed ends:
\begin{quote}
(Sec.\;2.2) We now consider the case where~$N$ is huge, say
$N=1,000,000$ or so.  Then for the lowest modes (say the first few
thousand), there are very large number of beads between each node.
Thus the displacement varies slowly from one bead to the next [We
shall not consider here the highest modes, since they approach the
``zigzag limit", where a description using a continuous function is
not possible.] (ibid., p.\;51)
\end{quote}
Furthermore,
\begin{quote}
(Sec.\;2.4) In sec.\;2.2 we considered a continuous string\ldots{} In
this section we will find the exact solutions for the modes of a
uniform beaded string having~$N$ beads and with fixed ends.  In the
limit that we take the numbers of beads~$N$ to to be infinite (and
maintain the finite length~$L$), we shall find the standing waves that
we studied in Sec.\;2.2.  Our purpose is not merely that, however.
Rather, we shall find that, in going to the limit of a continuous
string, we discarded some extremely interesting behavior of the
system. (ibid., p.\;72)
\end{quote}
These comments would be considered by traditionally trained
mathematicians as an ``informal discourse of a physicist.''  But for a
mathematician trained in nonstandard analysis, they can be easily
translated into a perfectly formalized mathematical text.

\subsection{The hyperfinite vibrating string}

We will now formalize Crawford's approach in Nelson's IST.
Let~$\Sigma_N$ be the following system of differential equations:
\begin{equation}
\label{NS}
\Sigma_N
\begin{cases}
\displaystyle \tfrac{dx_j}{dt} = y_j(t),
\quad\quad\quad\quad\quad\quad\quad\quad\quad\quad\quad\quad\; j =
1,\ldots, N-1\\ \tfrac{dy_j}{dt} = \tfrac{K^2}{h^2}(x_{j-1}(t) - 2
x_j(t) +x_{j+1}(t)),\; j = 1,\ldots, N-1\\ x_0(t) = x_N(t) = 0
\end{cases}
\end{equation}
where~$N \cdot h = 1$.

Assuming~$K = 1$, from elementary calculus we know that the solutions
are given by the following equations:

\begin{equation}\label{solution}
x_j(t) = \Sigma_{n = 1}^{N-1}u_n\cos(\omega_nt)+
\frac{v_n}{\omega_n}\sin(n\pi j h),
\end{equation}
where
\begin{equation}\label{const}
\begin{array}{lcl}
\omega_n &= & \displaystyle \frac{2}{h}\sin\Big(\frac{n\pi h}{2}\Big)
\\ [8pt] u_n& =& \displaystyle 2h \Sigma_{k = 1}^{N-1}x_n(0)\sin(k\pi
n h)\\[8pt] v_n &=& \displaystyle 2h \Sigma_{k =
1}^{N-1}y_n(0)\sin(k\pi n h)
\end{array}
\end{equation}

In classical mathematics one considers the continuous wave equation: 
\begin{equation}
\label{cont}
\Sigma \quad \quad \left\{
\begin{array}{lcl}
\displaystyle \frac{\partial ^2X}{\partial t^2}(z,t) &= &
\displaystyle \frac{\partial^2 X}{\partial z^2}(z,t)\\ [8pt] X(0,t) &=
&X(1,t) = 0\\
\end{array}
\right.
\end{equation}
to be ``the model'' and proves the existence of solutions~$(t,z)
\mapsto X(t,z)$, continuous with respect to~$t$ and~$z$ wich satisfy
(\ref{cont}); if one is interested in computer simulations it can be
proved that suitable solutions of the discrete (\ref{NS}) converge in
some sense to a solution of \eqref{cont} when~$N$ tends to infinity.
By the way, in classical mathematics, the system (\ref{NS}) is usually
viewed as \emph{an approximation of} (\ref{cont}).

But in nonstandard analysis we can consider (\ref{NS}) to be ``the
model'' provided that~$N$ is nonstandard (infinite).  Explicit
formulas (\ref{solution}) still stand.  Typically a physicist is less
interested in the \emph{existence} of solutions than in their
\emph{properties}.  For instance he might be interested in the
``shape" of the string at time~$t$, namely, in the aspect of the
``dotted line" in~$[0,1] \times \mathbb{R}$:
\begin{equation}
\label{dot}
\big\{ (n \,dz , u_n(t)) \colon n = 0,1, \ldots, N\big\}
\end{equation}
with~$dz = h$.  An important question is whether the ``dotted line''
{\em seems continuous}.  (The physicist asks ``whether the
displacement varies slowly from one bead to the next'').  The
mathematical formalization is the following.

\begin{definition}
The ``dotted line'' (\ref{dot}) is said to be \emph{S-continuous} if:
\[
k\;dz \approx 0 \Longrightarrow u_{n+ k}(t) \approx u_n(t).
\]
\end{definition}
The S-continuity of the initial condition~$x_n(0)$;~$n= 0, 1,\cdot
\cdot \cdot, N$ is insufficient to imply the S-continuity of the
solution.  What is required in addition is that the {\em initial
energy}
\[
E(0) = \frac{1}{2} \sum_{k = 1}^{N-1}
\left(\frac{x_{n+1}(0)-x_n(0)}{h}\right)^2h +\sum_{k = 1}^{N-1}y_n(0)
\]
be {\em limited}. In fact consider an initial condition such that
$x_n(0) = 0$; if an unlimited energy at time~$t=0$ is contained in the
unlimited modes~$\lambda_n$ it turns out that after some duration all
this energy will concentrate at some point and ``break'' the solution;
for details see Delfini--Lobry \cite{DL}.  It is one of the
interesting aspects of the behavior of the system that we can observe
when we do not discard high modes.  Notice that in the preceding lines
one uses only the \emph{Idealization axiom} of IST.

If one wishes to make the connection with the continuous model, one
uses the \emph{shadow} of the ``dotted line'' whose definition uses
the full strength of the \emph{Standardization} axiom.  Then it can be
proved under suitable asumptions on the initial conditions (including
limited energy) that the {\em shadow} of the ``dotted line'' is the
graph of a differentiable function~$\varphi(t,z)$ which is a solution
of (\ref{cont}).

We acknowledge that our example is very simple and that physicists did
not await Robinson to understand all the information contained in
formulas (\ref{solution}) and (\ref{const}).

\subsection{Hyperfinite Brownian motion}

This point of view was advocated in Nelson \cite{Ne87} (a less
elementary approach using Loeb measures was pioneered in Anderson
\cite{An76}), where Nelson defined the Brownian motion as a discrete
random walk of infinitesimal step size at times~$0, dt, 2dt,\ldots,
kdt,\ldots,Ndt = 1$, with~$dt$ infinitesimal. Namely, we set
~$$\xi_0
= 0;\quad\xi_{t+dt} = \xi_t + z_t \sqrt{dt}$$ where~$z_t$ is a
sequence of independent random variables taking the values~$+1$
or~$-1$ with probability~$\frac{1}{2}$.  The probability space for
such a process is just the (hyper)finite set~$\Omega =\{-1,+1\}^N$
with its~$\sigma$-algebra equal to~$ P(\Omega)$. A trajectory of this
Brownian motion is defined as the mapping~$t \mapsto x_t$ with domain
in what Nelson refers to as the near-interval, namely the hyperfinite
set~$\{0,dt,2dt,\ldots,kdt,\ldots,Ndt = 1\}$.  The important point
here is that despite its discrete definition, which fits well with the
intuition that ``at each instant one chooses at random'' (notice that
in the limit~$dt \rightarrow 0$ this concrete meaning is lost), the
notion of continuity is perfectly defined by {\em S-continuity} (see
above).  Then Nelson proves the following:
\begin{quote} 
\emph{Almost surely}, a trajectory of the Brownian motion is
S-continuous.
\end{quote} 
Here {\em almost surely} means that for every standard~$\varepsilon >
0$ the (external) set of trajectories that are not S-continuous is
contained in a (true) set of probability less than~$\varepsilon$. This
result is obtained within a very light subsystem of IST (just
Idealization is needed). The title of \emph{Radically Elementary
Probability Theory} delivers on its promise.  In fewer than 80 pages,
starting from the early beginnings including nonstandard analysis and
some basics of probability theory, the mathematical model for the
physical Brownian motion is constructed.  Moreover the last chapter
called \emph{The de
Moivre--Laplace--Lindeberg--Feller--Wiener--L\'evy--Doob--Erdos--Kac--Donsker--Prokhorov
theorem} provides a large amount of modern results on stochastic
processes.  It is the decision not to replace clear nonstandard
statements by wordier conventional paraphrases that makes things
elementary.  Let us quote Nelson's Preface, paragraphs 2 and 3:
\begin{quote}
This work is an attempt to lay new foundations for probability theory,
using a {\em tiny bit} of nonstandard analysis. The mathematical
background required is little more than which is taught in high
school, and it is my hope that it will make deep results from the
modern theory of stochastic processes readily available to anyone who
can add, multiply and reason.\\ What makes this possible {\em is the
decision} to leave the results in non-standard form. Nonstandard
analysts have a {\em new way of thinking} about mathematics, and if it
is not translated back into conventional terms then {\em it is seen to
be remarkably elementary.}  Nelson \cite[p.\;vii]{Ne87} (emphasis
ours)
\end{quote} 
At the next stage, in order to connect elementary results to classical
ones, the full IST system is used to prove the (formal) mathematical
equivalence of the hyperfinite model with the classical Wiener
process. Let us quote again from Nelson:
\begin{quote}
The purpose of this appendix is to demonstrate that theorems of the
conventional theory of stochastic processes can be derived from their
elementary analogues by arguments of the type usually described as
generalized nonsense; there is no probabilistic reasoning in this
appendix. This shows that the elementary nonstandard theory of
stochastic processes can be used to derive conventional
results\ldots{} \cite[p.\;80]{Ne87}
\end{quote}
Nelson goes on to make the following additional point:
\begin{quote}
\ldots{} on the other hand, it shows that neither the elaborate
machinery of the conventional theory nor the devices from the full
theory of nonstandard analysis, needed to prove the equivalence of the
elementary results with their conventional forms, add anything of
significance: the elementary theory has the same scientific content as
the conventional theory.  This is intended as a self destructing
appendix.%
\footnote{Nelson's cryptic comment calls for a clarification.  What
Nelson is apparently referring to is the fact that the conventional
formalism for random walks in particular and stochastic analysis in
general, as developed by Kolmogorov and others, relied on elaborate
machinery based on measure theory.  Nelson viewed the replacement of
elaborate machinery by radically elementary considerations as a
positive scientific development, which however has self-destructing
aspects as far as the work of the traditional practitioners themselves
is concerned.}
Nelson \cite{Ne87}
\end{quote}
This view regarding stochastic processes was extended to diffusion:
\[
\xi_0 = 0;\quad\xi_{t+dt} = \xi_t + f(\xi_t)dt+ z_t \sqrt{dt}
\]
in Beno\^\i t \cite{Be95}, to stochastic analysis as applied to
finance in van den Berg \cite{Va00}, and to complex physical particle
systems in Weisshaupt \cite{We09}, \cite{We11}.  In a similar vein,
Lobry \cite{LOB92}, Lobry--Sari \cite{LoSa04} advocate viewing ODEs
with discontinuous right hand side on the basis of a simple Euler
scheme with infinitesimal step.  The path of \emph{Radically
Elementary Modeling} was also pursued in Diener--Lobry \cite{DL85},
Harthong\;\cite{Ha84}, Fliess \cite{Fl06} and other works.

\section{More on Nelson}
\label{s6}

\epigraph{Bien que tous les math\'ematiciens ne le reconnaissent pas,
il existe une ``r\'ealit\'e math\'ematique archa\"\i que''.  Comme la
r\'ealit\'e du monde ext\'erieur, celle-ci est a priori non
organis\'ee, mais r\'esiste \`a l'exploration et r\'ev\`ele une
coh\'erence.  Non mat\'erielle, elle se situe hors de l'espace-temps.
A.\;Connes\;\cite{Co00d}}

\epigraph{The idea that there is a pure world of mathematical objects
(and perhaps other ideal objects) totally divorced from our
experience, which somehow exists by itself is obviously inherent
nonsense.  M.\;Atiyah \cite[p.\;38]{At}}

We first give a more detailed discussion of the II hypothesis outlined
in Section~\ref{s11}.

\subsection{Kleene, Wang, Putnam, and others on the II} 
\label{s23}

Many set theorists think of the semantics of set theory as given in
terms of an \emph{intended model} or \emph{intended interpretation},
in accordance with a realist philosophy of mathematics; for details
see Maddy \cite{Ma90}.  Intended interpretations concern not merely
the natural numbers but, more generally, purported entities in
something described as ``a mathematical practitioner's universe,''
``primordial mathematical reality'' (a term coined by A.\;Connes), or
the like.  Concerning such II hypotheses, Kleene wrote:
\begin{quote}
Since a formal system (usually) results in formalizing portions of
existing informal or semiformal mathematics, its symbols, formulas,
etc.  will have meaning or interpretations in terms of that informal
or semiformal mathematics.  These meanings together we call the
(\emph{intended} or \emph{usual} or \emph{standard})
\emph{interpretation} or \emph{interpretations} of the formal system.
Kleene \cite[p.\;200]{Kl67}
\end{quote}
In a similar vein, Hao Wang writes:
\begin{quote}
The originally intended, or standard, interpretation takes the
ordinary nonnegative integers~$\{0, 1, 2, \ldots\}$ as the domain, the
symbols~$0$ and~$1$ as denoting zero and one, and the symbols~$+$ and
$\cdot$ as standing for ordinary addition and multiplication (see
section ``Truth definition of the given language'' in Wang
\cite{Wa16}).
\end{quote}

An \emph{intended model} could be defined as one that ``reflects our
\emph{intuitions} [about natural numbers] adequately''
Quinon--Zdanowski \cite[p.\;313]{QZ} (emphasis added).  Haim Gaifman
writes:
\begin{quote}
Intended interpretations are closely related to realistic conceptions
of mathematical theories.  By subscribing to the standard model of
natural numbers, we are committing ourselves to the objective truth or
falsity of number-theoretic statements.  \ldots{} Realism and intended
interpretations are thus intimately related; often they are treated as
the same problem.  Gaifman \cite[p.\;15]{Ga04}
\end{quote}
Thus the II hypothesis entails that~$\N$ obtains a detailed reference
in the ordinary counting numbers.  Meanwhile, Robinson wrote:
``mathematical theories which, allegedly, deal with infinite
totalities do not have any detailed \ldots{} reference.''%
\footnote{A related point was made by Salanskis; see note~\ref{f37}.}
\cite[p.\;42]{Ro75}


Weber claimed to quote%
\footnote{The adage Weber reports in Kronecker's name is known not to
appear in any of Kronecker's writings; see Ewald \cite[p.\;942,
note\;{a}]{Ew96}.}
Leopold Kronecker as positing an allegedly immutable status of the
integers ``whereas everything else is the work of man.''  Weber's
quote is often misattributed to Kronecker himself.  However, Kronecker
specifically wrote, on the contrary, that numbers were a creation of
the human mind, and contrasted numbers that are human artefacts, with
space and time which he felt were outside the mind (possibly following
the traditional pre-Einsteinian philosophers):
\begin{quote}
The principal difference between geometry and mechanics on one hand,
and the other mathematical disciplines we comprehend under the name of
\emph{arithmetic}, consists according to Gauss in this: the object of
the latter, number, is a pure product of our mind, while space as well
as time has reality also outside of our mind which we cannot fully
prescribe a priori.  Kronecker \cite[p.\;339]{Kr87}
\end{quote}
See also Gauthier \cite[p.\,163]{Ga15}.  Hilary Putnam doubted our
ability to \emph{fix} an intended interpretation and seemed to treat
the~II as a throwback to Kantian \emph{noumena} \cite[p.\;482]{Pu80}.%
\footnote{Putnam's comment there about ``noumenal waifs'' indicates an
impatience with the typical post-Dedekind mathematician's assumption
that~$\N$ finds a detailed reference in the ordinary intuitive
counting numbers.  Putnam seems to view the latter as inaccessible
noumena, the mathematician's identification of~$\N$ with these noumena
as an unwarranted assumption, and the search for \emph{the} intended
interpretation as a futile search for parenthood for the said waiflike
noumena.  See further in note~\ref{f37}.}
For responses to Putnam see e.g., Horsten \cite{Ho01}, Gaifman
\cite{Ga04}.

\subsection{Predicate on the familiar real line}
\label{s72}

In the internal view of IST, rather than thinking of the standard sets
as being the familiar ones, embedded in a larger nonstandard world,
one essentially thinks of the nonstandard universe as the familiar
world (in the terminology of Section~\ref{s52}), with standard
structures being picked out of it by means of a suitable predicate.

On this view, one has the real numbers including both infinite and
infinitesimal reals, and one can say when two finite real numbers have
the same standard part, etc.  In this picture, we think of the
\emph{familiar} real line as what in the other picture would be the
nonstandard one, and then we have a predicate on that, which
corresponds to the range of the star map in the other approach.  So
some real numbers are standard, and some functions are standard and so
on.

Hamkins wrote: ``One sometimes sees this kind of perspective used in
arguments of finite combinatorics, where one casually considers the
case of an infinite integer or an infinitesimal rational.''  That kind
of talk may seem alien to cousin Georg (see Section~\ref{wings}), but
for those who adopt the picture it is useful.  In a sense, one goes
the whole nine yards into the nonstandard realm until it turns
\emph{familiar}.

There are several characterisations of Nelson's framework.  Each
characterisation employs a slightly different philosophical starting
point.  Peter Loeb's position as expressed in \cite[p.\;vii]{LW} is
that (1) Nelson is working exclusively in the nonstandard universe,
and (2)~there is no standard world in this setting.  This second point
is of course mathematically correct to the extent that the predicate
\st{} violates the axiom of separation.  As far as the first point is
concerned, it involves a bit of an equivocation on the meaning of the
word \emph{standard}:
\begin{enumerate}
\item[(a)] its technical meaning in the context of Robinson/Nelson,
and
\item[(b)] the meaning of \emph{ordinary/usual}.
\end{enumerate}
The equivocation is disambiguated in Section\;\ref{s52}.  In line with
the CD+II mindset (see Section\;\ref{s11}), cousin Georg (see
Section~\ref{wings}) would bridle at the idea that the ordinary/usual
real line should contain infinitesimals in any sense, and so Loeb is
correct in this sense, as well.  Other experts in nonstandard analysis
point out a complementary mathematical point that is also valid,
namely the following.

Nelson demonstrated that infinitesimals can be found within the
ordinary real line itself in the following sense.  Infinitesimals are
found in the real line by means of enriching the language through the
introduction of a unary predicate \st{} and postulating an axiom
schema (of Idealization), one of most immediate instances of which
implies the existence of infinitely large integers and hence nonzero
infinitesimals, as in his 1977 article; see Katz--Kutateladze
\cite{15c} for a related discussion.  Recall that Nelson's framework
\emph{Internal Set Theory} (IST) is a conservative extension of ZFC.\,
In other words, the entire package goes over, including each article
published in the \emph{Annals of Mathematics} so long as it makes no
use of Sarah's switch (see Section~\ref{wings}) \st.

To put it more colorfully, infinitesimals were there all along, but
cousin Georg hasn't noticed them.

\subsection{Multiverse}
\label{s62}

Hamkins \cite{Ha09}, \cite{Ha11} proposes a view of the foundations of
mathematics where there are many distinct concepts of set, each
instantiated in the corresponding set-theoretic universe These works
formulate a number of principles that the multiverse should satisfy.
An interesting observation for our purposes is that the following
principle is compatible with Hamkins's multiverse.
\begin{quote}
\emph{Principle B:} For every \text{ZFC} universe~$V$ in the
multiverse there is a class~$S\subseteq V$ such that all the axioms of
the theory \text{BST} hold in~$(V, \in, S)$.
\end{quote}

In other words, given any universe~$V$ and a class~$S$ as above, the
predicate~$\st=\st_{V,S}^{\phantom{I}}$ defined by
\[
\st(x) \text{ if and only if } x\in S
\]
distinguishes some elements of~$V$ as standard, in a way that makes
BST hold, without adding any new elements to~$V$.  Thus Nelson's claim
that the infinitesimals have been there all along without being
noticed (see Section~\ref{s72}) is literally true in this framework,
though in the context of BST rather than Nelson's original theory IST.
For the benefit of readers familiar with model theory and BST, we give
some technical underpinnings of these claims.  Victoria Gitman and
Hamkins construct a ``toy'' model of the multiverse axioms in ZFC;
see~\cite{GH}.  We show that Principle\;B holds in this model.

\begin{theorem}
The Gitman--Hamkins model satisfies the \emph{Principle\;B}.
\end{theorem}

\begin{proof}[Outline of proof]
One can start with the fact that a standard core interpretation for
BST is definable in ZFC (Kanovei--Reeken \cite[Theorems~4.1.10(i) and
4.3.13]{KR}), where the expression ``standard core'' indicates that
the universe~$V$ of ZFC is isomorphic to the universe of standard sets
in the interpretation.

An immediate corollary \cite[Corollary 4.3.14]{KR} is that for every
model
$(M, \in_M)$ of ZFC there is a corresponding model~$(N,\in_N,S_N)$ of
BST, in which the class of standard sets~$S_N$ is isomorphic to~$M$.

The multiverse of Gitman and Hamkins consists of all countable
computably saturated models of ZFC.  It is easy to see that if
$(M,\in_M)$ is countable and computably saturated, then also the
corresponding 
$(N, \in_N, S_N)$ is countable computably saturated.
Furthermore,~$(N, \in_N)$ and~$(M, \in_M)$ have the same theory and
the same standard system.  Hence they are isomorphic
\cite[Key\;Lemma\;6]{GH}.  If~$S$ is the image of~$S_N$ under such an
isomorphism, then~$(N, \in_N, S_N)$ and~$(M, \in_M, S)$ are also
isomorphic.  In particular,~$(M,\in_M,S)$ satisfies BST.
\end{proof}

It is not known whether BST can be replaced by Nelson's IST in this
argument; see the related discussion in Section~\ref{s52}.
Essentially the same argument works for the theory GRIST (see
Section~\ref{s57}) starting with the fact that a standard core
interpretation for GRIST is definable in ZFC, Hrbacek \cite{Hr12}.

\begin{theorem}
One can consistently assume that every universe of \emph{ZFC} in the
multiverse can be stratified into many levels of relative standardness
by a binary relation~$\sqsubseteq$ in such a way that all the axioms
of \emph{GRIST} hold in~$(V,\in,\sqsubseteq)$.
\end{theorem}

\subsection{Switches}
\label{s63}

Hamkins \cite{Ha12}, \cite{Ha15} discusses the concept of
\emph{switches}, exemplified by the continuum hypothesis (CH).  Here
the terminology of switch implies that it can be turned on or off at
will.  Similarly, using modern techniques of forcing, one can pass
from a model of ZFC that satisfies CH to one that satisfies~$\neg$\,CH
and vice versa.

The availability of the predicate \st{} in the language is also a type
of switch.  Namely, all models of ZFC can be built up into models of
BST, turning the switch on as it were.  Meanwhile, the forgetful
functor (just forget about the switch in the glove compartment) takes
you from BST back to ZFC.

The analogy with switches is that one model has infinitesimals, while
the other doesn't.  This is analogous to one model possessing a set of
intermediate cardinality between~$\N$ and~$\N^{\N}$, and the other
not.  Both switches arguably challenge the II hypothesis; see
Section~\ref{s11}.

Now \st{} is a different type of switch, but whatever the situation
may be for the ontology and epistemology of switches, Sarah's switch
implies (not merely that one can fly faster than drive but) that the
discovery of the switch leaves all theorems in place, so that Nelson's
real numbers can be viewed as the familiar ones, and the
infinitesimals have been there all along without being noticed--even
though the familiar model from the last third of the 19th century
apparently didn't have them.

\subsection{Historical antecedents}
\label{s85}

The view that numbers come in more than one flavor is closely parallel
to the dichotomy of \emph{assignable} vs \emph{inassignable}
quantities (the latter being viewed as \emph{useful fictions}) in
Leibnizian calculus; see Bair et al.\;\cite{13a} and Bascelli et
al.\;\cite{16a} for more details.  Even earlier, Fermat's technique
of \emph{adequality} exploited procedures using~$E$ in a striking
anticipation of later infinitesimal techniques; see \cite{13e}.
Meanwhile L.\;Carnot spoke of \emph{quantit\'es d\'esign\'ees} and
\emph{quantit\'es auxiliaires} in 1797; see Barreau \cite[p.\;46,
53]{Ba89}.

Stolz, du Bois-Reymond, and others were working on infinitesimals at
the end of the 19th century, and had they joined forces with Frege or
Peano to conceive an axiomatisation that would actually incorporate
Leibniz's assignable \emph{vs} inassignable distinction, 20th century
mathematics may have looked different.  This viewpoint involves being
able to conceive of the history of mathematics as something other than
inevitable march toward Weierstrassian epsilontics and
Cantor--Dedekind reals.  In more precise philosophical terms, it
involves envisioning the history of mathematics in terms of the
\emph{Latin model} rather than the \emph{butterfly model}, to borrow
the terminology from Ian Hacking \cite{Ha14}.

To comment on Hacking's distinction between the \emph{butterfly model}
and the \emph{Latin model}, we note the contrast between a model of a
deterministic biological development of animals like butterflies, as
opposed to a model of a contingent historical evolution of languages
like Latin.

Hacking's dichotomy applies to the development of the field of
mathematics as a whole.  Some scholars view the development of
mathematics as a type of organic process predetermined genetically
from the start, even though the evolution of the field may undergo
apparently sudden and dramatic changes, like the development of a
butterfly which passes via a cocoon stage which is entirely unlike
what it is pre-destined to produce.

The Latin model acknowledges contingent factors in the development of
an exact science (mathematics included), and envisions the possibility
of other paths of development that may have been followed.  For
example, had an axiomatic formalisation of infinitesimals been
proposed earlier, it might have been incorporated into the early
formalisations of set theory, and spared us the inanity of the
Cantor--Russell anti-infinitesimal vitriol, reflecting the state of
affairs in mathematical foundations during the second half of the 19th
century; for additional details see \cite{13f}.

Hacking's perspective is at odds with some of the received history of
mathematical analysis.  A related point is made by P.\;Mancosu in the
following terms:
\begin{quote}
the literature on infinity is replete with such `Whig' history.
Praise and blame are passed depending on whether or not an author
might have anticipated Cantor and naturally this leads to a completely
anachronistic reading of many of the medieval and later contributions.
Mancosu \cite[p.\;626]{Ma09}
\end{quote}
In his critique of intuitionism, Bernays introduced a distinction
related to \emph{assignable/inassignable} in terms of
\emph{accessible} vs \emph{inaccessible}:

\begin{quote}
Brouwer appeals to intuition, but one can doubt that the evidence for
it really is intuitive. Isn't this rather an application of the
general method of analogy, consisting in extending to inaccessible
numbers the relations which we can concretely verify for accessible
numbers? As a matter of fact, the reason for applying this analogy is
strengthened by the fact that there is no precise boundary between the
numbers which are accessible and those which are not.  Bernays
\cite{Be35} (translation by Charles Parsons)
\end{quote}
Thus, we obtain infinitesimals as soon as we assume that (1) there are
assignable/standard real numbers, that obey the same rules as all the
real numbers, and~(2)~there are real numbers that are not assignable.
Of course, neither Leibniz nor Carnot employed any set-theoretic
notions to specify an \emph{ontology} of their infinitesimals, but
their \emph{procedures} find better proxies in modern infinitesimal
frameworks than in modern Weierstrassian ones.%
\footnote{The procedures \emph{vs} ontology distinction is dealt with
in greater detail in the articles Borovik--Katz \cite{12b}, Bair et
al.\;\cite{17a}, Bascelli et al.\;\cite{16a}, B\l aszczyk et
al.\;\cite{17c}, B\l aszczyk et al.\;\cite{17d}.}
In more technical terms, one considers the ordinary ZFC formulated in
first order logic, and adds to it the unary predicate \st{} and the
axiom schemata.

Robinson's framework has enabled a reappraisal of the procedures of
the pioneers of infinitesimal analysis.  For a broad outline of such a
program see the studies \cite{13a} (2013), \cite{13b} (2013),
\cite{14a} (2014), \cite{17d} (2017).  Specific scholars studied
include
\begin{itemize}
\item
Stevin in \cite{12c} (2012);
\item
Gregory
in \cite{17b}; 
\item
Fermat in \cite{13e} (2013);
\item
Leibniz in \cite{12e} (2012), \cite{13f} (2013), \cite{14c} (2014),
\cite{16a} (2016), \cite{17c}\;(2017);
\item
Euler in \cite{15b} (2015), \cite{17a}\;(2017);
\item 
Cauchy in (Laugwitz \cite{La87}) and in \cite{11b} (2011), \cite{12b}
(2012).
\end{itemize}

\section
{Reeb, naive integers, and \emph{Claim\,Q}}
\label{s7}

Georges Reeb's position regarding the use of infinitesimals was
developed in his essay entitled \emph{La math\'ematique non standard
vieille de soixante ans?}  There are two distinct versions of the
essay, \cite{Re79} (1979) and \cite{Re81} (1981).  The 1979 version of
the essay was reprinted in Salanskis \cite{Sa99}.  The book gives an
account of the \emph{constructivist view} of Reeb.  Reeb was closely
associated with J.\;Harthong in his philosophy of mathematics.  They
published the article \emph{Intuitionnisme\;84} \cite{HR} in the book
Barreau--Harthong \cite{BH89}.  One can distinguish at least three
distinct approaches to the problem of motivating a non-naive integer
in Reeb, as detailed in the next three sections.  See also Lobry
\cite{Lo89}, Diener--Diener \cite[p.\;4]{DD}.

\subsection{Nonformalizable intuitions}
\label{s71}

Reeb's position concerning \emph{naive integers} can perhaps be
described as follows.  The \emph{naive integers} are those integers
that all members of humanity share before they understand any advanced
mathematics.  Such \emph{naive integers} are already present in any
formal language since language is necessarily a \emph{succession} of
symbols, rather than an unordered collection of symbols.

If one accepts a formal language (e.g., that of ZF) with respect to
which~$\mathbb{N}$ is defined, one must accept also that one
understands the following \emph{informal} reasoning:
\begin{itemize}
\item We recognize ``$1$'' of~$\mathbb{N}$ as the \emph{naive}
\textbf{one};
\item We recognize ``$2$'' of~$\mathbb{N}$ as the \emph{naive}
\textbf{two};
\item etc.
\item If we recognize ``$n$'' of~$\mathbb{N}$ as the \emph{naive}
\textbf{n} then we recognize ``$n+1$'' of~$\mathbb{N}$ as the naive
\textbf{n\;plus\;one}.
\end{itemize}
Reeb then argues that an assertion to the effect that every element
of~$\N$ is naive is not supported by any formal mathematics, and
arrives at his
\begin{quote}
``Claim Q'': The naive integers don't fill up~$\mathbb{N}$.%
\footnote{\label{f37}Reeb's term in the original French was
\emph{Constat\;Q} which we loosely translate as \emph{Claim\;Q} fully
aware of the inadequacies of such a translation.  The difficulty of
the term was analyzed by Salanskis, who noted: ``Le probl\`eme du
remplissement de~$\N$ par les naifs est totalement d\'enu\'e de sens
si l'on ne joue pas le jeu de r\^ever que les formalismes suscitent
des r\'ef\'erents.''  Salanskis \cite{Sa94} (translation: ``The
problem of filling~$\N$ by the naive integers is totally meaningless
if one is not playing the game of dreaming that formalisms generate
referents.'')}
\end{quote} 
Thus claim Q is at tension with the CD+II mindset (see
Section~\ref{s11}).  

To anyone familiar with model theory, Reeb's ``claim Q" could easily
be interpreted in terms of the existence of nonstandard models of the
natural numbers, whether in PA or ZF, first constructed in
Skolem\;\cite{Sk33}.  However, Reeb \cite{Re79} takes a more
``fundamental" attitude and seems to argue for his ``claim Q" somehow
\emph{from first principles}, the naive integers being taken to be
available before a commitment to formal mathematics.  Reeb comments
that being \emph{naive} is not a mathematical concept, hence not
formalizable; hence they cannot be said to fill up~$\N$.

\subsection{Ideal intruders}
 
In \cite{Re81}, Reeb takes a different tack, and argues that any
infinitary construction in mathematics, as recognized at least since
Hilbert \cite{Hi26}, necessarily introduces \emph{ideal} elements not
intended by initial naive intuitions:
\begin{quote}
D. Hilbert - pour nous en tenir \`a un seul nom - a montr\'e
clairement, dans son c\'el\`ebre article sur l'infini, comment la
formalisation (par exemple~$\N , \R, \ldots$) de notions concr\`etes
(les entiers de tout le monde, les points du continu intuitifs)
introduit n\'ecessairement, en quelque sorte contre la vigilance du
formalisateur d'abondants objets id\'eaux, non d\'esir\'es.  Reeb
\cite[p.\;149]{Re81}
\end{quote}
Reeb goes on to refer to such elements as ``des intrus id\'eaux
in\'evitables lors de la formalisation.''

\subsection
{Link between intuitionism and nonstandard analysis}

In the book Diener--Reeb \cite[Chapter~9]{DR}, Reeb provides yet
another account of a non-naive integer in an intuitionistic setting,
in terms of the size of a hypothetical solution to~$x^n+y^n=z^n$ (this
was before A.\,Wiles; of course Fermat's last theorem can be replaced
by a conjecture that is still open).

Surprising as it may seem to the uninitiated, the French school of
nonstandard analysis draws a direct connection between Brouwer's
intuitionism and nonstandard analysis.  We refer to Diener--Reeb
\cite[Chapter 9]{DR} and Harthong--Reeb \cite{HR} for a detailed
discussion, while we sketch the motivation for this connection as
follows.

L.\;E.\;J.\;Brouwer was the founder of \emph{intuitionism}, the first
school of constructive mathematics; the latter aims to provide a
\emph{computational} foundation for mathematics based on the
Brouwer--Heyting--Kolmogorov (BHK) interpretation of logic.

To identify a part of classical mathematics as not-acceptable in
intuitionistic mathematics, Brouwer introduced a technique which would
later be called ``Brouwerian counterexamples;" see Mandelkern
\cite{Ma89} for an overview.

Such counterexamples come in \emph{strong} and \emph{weak} flavors and
are meant to cast doubt on the constructive/intuitionistic
acceptability of a given theorem or axiom.  A weak Brouwerian
counterexample against the law of excluded middle (LEM) is as follows:

\begin{quote}
$A \vee\neg A$ is not acceptable (under the BHK interpretation)
because there is no algorithm to decide whether $A=$ `Goldbach's
conjecture' is false or not (or any as-yet unproved conjecture).
\end{quote}

However, Goldbach's conjecture has been verified using computers up to
very large numbers.  Thus, there is the possibility, discussed at
length in \cite{HR}, that Goldbach's conjecture is false, but that the
counterexample cannot be constructed (in principle and/or in
practice).  Hence, Goldbach's conjecture would be ``true in the real
world", but false in principle.  Here, the real world could be either
``the physical world" or the world of constructive mathematics,
according to \cite{HR}.

In other words, all the natural numbers one can construct (in
principle or in practice) do satisfy Goldbach's conjecture, but
``there are out there" numbers which do not.  Following \cite{HR}, one
could refer to the former (constructive) numbers as \emph{naive
integers} (satisfying Goldbach's conjecture) and to the others as
\emph{non-naive}.  Now, Goldbach's conjecture is just one example, and
there will always be unsolved conjectures, so the previous idea is
persistent in that sense.  The above reasoning is how one could
interpret the adage by Reeb and Harthong that ``the naive integers do
not fill up $\N$" based on Brouwer's weak counterexamples to LEM.

Meanwhile, the observation that ``the naive integers do not fill
up~$\N$" is the basic motivation for having nonstandard numbers in
\cite{HR}, and we observe how Brouwer's intuitionism motivates the
existence of nonstandard numbers, especially in the sense of Nelson's
IST.

\subsection{Nelson-style motivations}

There are Nelson-style motivations for such a \emph{claim\;Q} but they
don't seem to be the same as Reeb's.  There are various motivations
for why naive integers shouldn't exhaust~$\N$.  One of them is in
terms of a multitude that's too vast to be expressed by even a
computer the size of the universe running the entire time allotted to
our civilisation and exploiting the fastest growing functions in our
logical arsenal including \emph{super}busy beavers.  Such a number
could function as infinite for all practical purposes at the
\emph{naive} level.  This indicates a useful lack of homogeneity
of~$\N$ and a promise of a richer structure which is better captured
in terms of an enriched syntax as in Nelson's system, which singles
out \emph{standard} (or \emph{assignable}) elements out of~$\N$ by
means of a single-place predicate violating the separation axiom.

\subsection{Quantum intuitions}

The breakdown of infinite divisibility at quantum scales poses an
undeniable challenge to the CD+II mentality (see Section~\ref{s11}).
It makes physically irrelevant a literal reading of the mathematical
definition of the derivative in terms of limits as~$\Delta x$
\emph{tends to zero}, since attempting to calculate the derivative for
increments below the Planck scale would yield physically meaningless
results.  Rather, quotients like~$\frac{\Delta y}{\Delta x}$ need to
be taken in a certain range, or at a suitable level.  Attempts to
formalize Planck's~$\hslash$ as an infinitesimal go back at least to
Harthong \cite{Ha84}, Werner--Wolff \cite{WW}.  The article
Nowik--Katz \cite{15d}%
\footnote{\label{f24}In his technical report for the CNRS in 1985,
Alain Connes wrote: ``La s\'eduction de l'analyse non standard est due
en grande partie \`a la cr\'eation d'un vocabulaire suggestif; la
g\'eom\'etrie diff\'erentielle s'est bien gard\'ee de c\'eder \`a
cette tentation, etc.''  (cited in Barreau \cite[p.\;35]{Bar94}).  Had
we been aware of Alain Connes's \emph{seductive} comment at the time
we might have subtitled our article \cite{15d} ``Differential geometry
reduced at last.''}
developed a general framework for differential geometry at
level~$\lambda$.  In the technical implementation~$\lambda$ is an
infinitesimal but the formalism is a better mathematical proxy for a
situation where infinite divisibility fails for physical reasons, and
a scale for calculations needs to be fixed accordingly.

If one accepts Reeb's ``Claim Q,'' then any \emph{non-naive} integer
can be described as \emph{infinite} in the sense that it is greater
than any \emph{naive} integer.  Once one forms the field of fractions,
the inverse of such an \emph{infinite} integer in this sense becomes
an \emph{infinitesimal}.

\section{Conclusion}

While mathematicians may not always think of set theory as being
\emph{dynamic}, such a switch in foundational thinking is helpful in
appreciating the advantages of the framework developed by Robinson,
Nelson, and others.


Just as the automobiles have gotten better over the past 130 years, so
also there is room for improvement as far as the 1870 set-theoretic
foundations are concerned.  The automobile industry is dynamic and
flexible willy-nilly (if you don't innovate, your competitors will)
but the received, and outdated, views on set theory seem to have
cornered the market like an \emph{intended} monopoly, alluded to in
Section~\ref{s23}.  Automobiles today have numerous options, from fuel
injection, 4-wheel drive, convertible, and GPS to more futuristic ones
like lift-off abilities to various degrees, with broad agreement as to
the general utility of such options some of which have become part of
the standard package.  We argue in favor of inclusion of the
infinitesimal option as part of the core package of set-theoretic
foundations.

A natural place to start would be in education, so as to restore
infinitesimals to the calculus curriculum, as in Cauchy's classroom at
the \emph{Ecole Polytechnique}.  Accordingly, over the past few years
we have trained over 400 freshmen using Keisler's infinitesimal
calculus textbook \cite{Ke86}, and summarized the results in the study
Katz--Polev \cite{17h}.  At the high school level infinitesimal
calculus has been taught in Geneva for the past twelve years, based on
the approach developed in \cite{HLO}.

In addition to serving as a fruitful tool for what Leibniz called the
\emph{Ars inveniendi}, Robinson's framework has occasioned a deepened
reflection on mathematical foundations in general and the meaning of
\emph{number} in particular.

\section*{Acknowledgments}

We are grateful to Eric Leichtnam and Dalibor Pra\v z\'ak for helpful
comments on an earlier version of the manuscript.  M.\;Katz was
partially supported by Israel Science Foundation grant no.\;1517/12.
S.\;Sanders was supported by the following funding bodies: FWO
Flanders, the John Templeton Foundation, the Alexander von Humboldt
Foundation, and the Japan Society for the Promotion of Science, and
expresses gratitude towards these institutions.

\end{document}